\documentclass[11pt,a4paper]{cedram-aif}
\usepackage[utf8]{inputenc}
\usepackage{ucs}
\usepackage{amsmath}
\usepackage{amsthm}
\usepackage{amsfonts,upgreek}
\usepackage{amssymb,mathrsfs}
\usepackage[english]{babel}
\usepackage[cyr]{aeguill}
\usepackage{enumerate}
\usepackage{geometry}
\usepackage{blkarray}
\usepackage[bookmarks=true]{hyperref}
\hypersetup{colorlinks=true,linkcolor=red,citecolor=blue}
\usepackage{pgf,tikz}
\usetikzlibrary{arrows}
\usepackage{yhmath}

\newtheorem{question}{Question}

\newcommand{\R}{\mathbb{R}}

\newcommand{\N}{\mathbb{N}}
\newcommand{\C}{\mathbb{C}}

\newcommand{\K}{\mathbb{K}}

\DeclareMathOperator{\trace}{trace}

\newcommand{\cpangle}{\overline{\angle}}

\numberwithin{equation}{section}
\renewcommand{\epsilon}{\varepsilon}

\DeclareMathOperator{\Wedge}{\wedge^2}

\begin{document}
	
\title[Infinite dimensional Riemannian symmetric spaces]{Infinite dimensional Riemannian symmetric spaces\\
with fixed-sign curvature operator}
\author{\firstname{Bruno} \lastname{Duchesne}}
\thanks{The author was supported by a starting grant from the Swiss National Foundation.}
\email{bruno.duchesne@ens-lyon.org}
\address{Einstein Institute of Mathematics\\ Edmond J. Safra Campus, Givat Ram\\The Hebrew University of Jerusalem\\Jerusalem, 91904, Israel\\ \\ \textit{current address}$\colon$\\  Institut Elie Cartan\\ Universit\'e de Lorraine\\  54506 Vandoeuvre-l\`es-Nancy, France }
\maketitle
\begin{abstract} We associate to any Riemannian symmetric space (of finite or infinite dimension) a L$^*$-algebra, under the assumption that the curvature operator has a fixed sign. L$^*$-algebras are Lie algebras with a pleasant Hilbert space structure. The L$^*$-algebra that we construct is a complete local isomorphism invariant and allows us to  classify Riemannian symmetric spaces with fixed-sign curvature operator. The case of nonpositive curvature is emphasized.
\end{abstract}
\section{Introduction}
\subsection{Riemannian symmetric spaces}
At the very end of the nineteenth century and during the beginning of the twentieth century, E. Cartan did a famous work of classification. He began by completing the proof (by W. Killing) of the classification of complex semisimple Lie algebras during his Ph.D. thesis and he continued by classifying real semisimple Lie algebras. Some years later, he introduced the so-called Riemannian symmetric spaces (``\textit{Une classe remarquable d'espaces de Riemann}") and classified them. The classification of symmetric spaces was reminiscent of the classification of real forms of complex semisimple Lie algebras (see \cite{MR1847105}).

Infinite dimensional differential geometry grew up from the twentieth century (see \cite{MR0203742} for an outline of the theory in the sixties and see \cite{MR1666820} for a more recent exposition) and it is not difficult to define when a Riemannian manifold, that is a manifold modeled on a separable Hilbert space with a Riemannian metric, is a symmetric space. Let $(M,g)$ be a Riemannian manifold, a \textit{symmetry} at a point $p$ is an involutive isometry $\sigma_p\colon M\to M$ such that $\sigma_p(p)=p$ and the differential at $p$ is -Id. A \textit{Riemannian symmetric  space} is a Riemannian manifold such that, at each point, there exists a symmetry.

An idea to classify these spaces could be to associate a ``semisimple" Lie algebra to them, to classify infinite dimensional semisimple Lie algebras and then return to symmetric spaces. We do not know a general classification of infinite dimensional Lie algebras nor a good notion of semisimple Lie algebras. Nonetheless, there is a remarkable exception to this lack of classification. R. Schue introduced complex L$^*$-algebras (Lie algebras with a compatible structure of Hilbert space, see Section \ref{L$^*$-algebras}) and classified the separable ones in \cite{MR0117575,MR0133408}. Later, independently, V.K. Balachandran \cite{MR0347920}, P. de la Harpe \cite{MR0282218} and I. Unsain \cite{MR0325721} classified separable real L$^*$-algebras.

Each L$^*$-algebra is an orthogonal sum of an abelian ideal and a semisimple ideal. Each separable semisimple L$^*$-algebra is a Hilbertian sum of simple ones. The simple L$^*$-algebras of infinite dimension belong to a finite list with three infinite families. They are closure of an increasing union of simple Lie algebras of finite dimension and classical type.

Unfortunately, the Lie algebra of the isometry group of a Riemannian symmetric space has no reason to be a L$^*$-algebra. For example, consider the Riemannian symmetric space $P^2(\infty)\simeq GL^2_\infty(\R)/O^2(\infty)$, that is the space of positive invertible operators of some separable real Hilbert space, which are Hilbert-Schmidt perturbations of the identity. This space is an infinite dimensional generalization of the symmetric space $SL_n(\R)/SO_n(\R)$ (See \cite[III.2]{MR0476820} and \cite{MR2373944}). The full orthogonal group $O(\infty)$ acts isometrically by conjugation on $P^2(\infty)$. In particular, the Lie algebra of all bounded skew-symmetric operators is a subalgebra of the Lie algebra of the isometry group. It is naturally a Banach Lie algebra but not a L$^*$-algebra.
\begin{rema}Michael Klotz proved in \cite[Theorem 5.24]{Klotz:2009qy} that any connected Banach symmetric space $M$ is an homogeneous space $G/K$ where $G$ is the group of automorphisms of $M$ and $K$ is a Banach-Lie Group. This result legitimizes the definition of Riemannian symmetric spaces that appears in \cite{MR2456274}. Moreover, it seems to be known that the isometry group of a Riemannian space is a Banach-Lie group but we do not know any reference. In the sequel, we do not use such result and the Lie algebra of Killing fields will play the role of the Lie algebra of the isometry group. In finite dimension, the Lie algebra of the isometry group of a Riemannian symmetric space and the algebra of Killing fields are naturally isomorphic.
\end{rema}
In the following theorem, we show that if one looks at a smaller (but large enough to encode the Riemann tensor) Lie algebra, one can find a L$^*$-algebra. We refer to section \ref{construction} for the definition of the curvature operator. 
\begin{theorem}\label{existence}Let $(M,g)$ be a simply-connected Riemannian symmetric space and let $p$ be a point in $M$. If $M$ has a fixed-sign curvature operator then there exists a real L$^*$-algebra $L$ with an orthogonal decomposition of Hilbert spaces
\[L=\mathfrak{k}\oplus\mathfrak{p}\] which has the following properties :
\begin{enumerate}[(i)]
\item the subspace $\mathfrak{k}$ is a L$^*$-subalgebra of $L$ and $\mathfrak{p}$ is isometric to the tangent space $T_pM$,
\item the Lie algebra generated by $\mathfrak{p}$ is dense in $L$ and is isomorphic to a subalgebra of the Lie algebra of Killing fields on $M$. 
\end{enumerate} 
\end{theorem}
The L$^*$-algebra obtained in Theorem \ref{existence} is the only one which satisfies properties (i) and (ii) (see Lemma \ref{uniqueness}). We call it a L$^*$-\textit{algebra associated to} $(M,g)$. The universal cover of a Riemannian symmetric space is a Riemannian symmetric space too (Proposition \ref{universal}). The L$^*$-algebra constructed allows us to give a complete description of  Riemannian symmetric spaces with fixed-sign curvature operator up to local isomorphism.
\begin{theorem}\label{localisomorphism}
Let $(M,g)$ and $(M',g')$ be Riemannian symmetric spaces with fixed-sign curvature operator. Let $L,L'$ be L$^*$-algebras associated to the universal covers $\widetilde{M}$ and $\widetilde{M'}$ as in Theorem \ref{existence}.

If there exists an isomorphism  of L$^*$-algebras between $L$ and $L'$ which intertwines the orthogonal decompositions $L=\mathfrak{k}\oplus\mathfrak{p}$ and $L'=\mathfrak{k}'\oplus\mathfrak{p}'$ then $M$ and $M'$ are locally isomorphic.
\end{theorem}
 If the curvature operator of a Riemannian manifold is nonpositive (respectively nonnegative) then the sectional curvature is nonpositive (respectively nonnegative) but the converse is false in general (See, e.g., \cite[§1.3]{MR0454884}). In finite dimension, a Riemannian symmetric space has nonpositive (respectively nonnegative) curvature operator if and only if it has nonpositive (respectively nonnegative) sectional curvature. This fact holds because the Riemann tensor is encoded in the Killing form of the Lie algebra of the isometry group (See \cite[Theorem 6]{MR0148010}, \cite[Section 4]{MR0454884} or Equation (\ref{Killing})). The main idea of this paper is to construct an analog of the Killing form starting from the Riemann tensor. It is natural to ask whether fix-sign sectional curvature implies fix-sign curvature operator, in infinite dimension too. More generally, we have the following question.
\begin{question}\label{question}
Is it true that for any Riemannian symmetric space, there is an orthogonal decomposition of the tangent space $\mathfrak{p}=\mathfrak{p}_-\oplus\mathfrak{p}_0\oplus\mathfrak{p}_+$ such that $\mathfrak{p}_-,\mathfrak{p}_0$ and $\mathfrak{p}_+$ are commuting Lie triple systems and the restrictions of the curvature operator is nonnegative on $\mathfrak{p}_-$, vanishes on $\mathfrak{p}_0$ and is nonpositive on $\mathfrak{p}_+$ ?
\end{question}
A positive answer to this question would imply a complete classification of simply-connected separable Riemannian symmetric spaces --- that is without any assumption on the curvature operator. Actually, if a Riemannian symmetric space has a dense increasing sequence of totally geodesic subspaces of finite dimension then Proposition \ref{union} shows that the answer to the above question is positive. Moreover, subsequent theorems will show that such a decomposition of the tangent space will imply the existence of a dense increasing sequence of totally geodesic subspaces of finite dimension.

To decompose Riemannian symmetric spaces in irreducible ones, we use Hilbertian products.
\begin{defi} Let $(X_i,d_i)$ be a countable family of metric spaces with base points $x_i\in X_i$. The product $\prod_i^2X_i$ is defined to be the set of elements $y=(y_i)$ of the Cartesian product of $X_i$'s such that $\sum d(x_i,y_i)^2\langle\infty$ and the distance between $y=(y_i)$ and $z=(z_i)$ is defined by $d(y,z)^2=\sum d(y_i,z_i)^2$.  This metric space is called the \textit{Hilbertian product} of the spaces $X_i$.
\end{defi} 
This definition depends on the choice of base points but if each $X_i$ has a transitive group of isometries then the product $\prod_i^2X_i$ does not depend on this choice (up to isometry). Moreover, this product space is complete if and only if each $(X_i,d_i)$ is so.
\begin{rema}\label{bounded curvature}In general, there is no notion (in the category of Riemannian manifolds) of Hilbertian product of Riemannian manifolds. The sectional curvature at each point has to be bounded (the Riemann 4-tensor at each point is continuous \cite[Proposition IX.1.1]{MR1666820}  and thus the sectional curvature is bounded). For example, the Hilbertian product of hyperbolic plans of curvature $-n$  cannot be a Riemannian manifold such that each hyperbolic space embeds as a totally geodesic submanifold.
\end{rema}

Technics that we used in nonpositive curvature and nonnegative curvature are slightly different. In nonpositive curvature the Cartan-Hadamard theorem  simplifies the classification and we give this simpler proof even if the technics used in nonnegative curvature are more general.
\subsection{Nonpositive curvature}
%
%
%
%
\begin{defi}A Riemannian manifold $(M,g)$ has no Euclidean local de Rham factor if its universal cover cannot be decomposed as a product $\mathcal{H}\times N$ where $\mathcal{H}$ is an Hilbert space of positive dimension and $N$ is an other Riemannian manifold.
\end{defi}
\begin{theorem}\label{noncompact} Let $(M,g)$ be a separable Riemannian symmetric space with nonpositive curvature operator and no Euclidean local de Rham factor. Then $(M,g)$ is isometric to a Hilbertian product 
\[M\simeq{\prod_i}^2 M_i\]
where  each $M_i$ is an irreducible finite dimensional Riemannian symmetric  space of noncompact type or is homothetic to an element of the following list :
\[GL_\infty^2(\R)/O^2(\infty),\quad U^{*\ 2}(\infty)/Sp^2(\infty),\quad U^2(p,\infty)/U^2(p)\times U^2(\infty),
\quad O^2(p,\infty)/O^2(p)\times O^2(\infty)\]
\[O^{*\ 2}(\infty)/U^2(\infty),\quad Sp^2_\infty(\R)/U^2(\infty),\quad Sp^2(p,\infty)/Sp^2(p)\times Sp^2(\infty), \]
\[GL_\infty^2(\C)/U^2(\infty),\quad O^2_\infty(\C)/O^2(\infty),\quad Sp_\infty^2(\C)/Sp^2(\infty)\]
where $p\in\N\cup\{\infty\}$.
\end{theorem}
The elements of the previous list are hence the irreducible infinite dimensional Riemannian symmetric  spaces with nonpositive curvature operator. Their construction is described in Section \ref{constructionspace}.

\begin{rema}If $M$ is a simply-connected symmetric space with nonpositive curvature operator then $M$ is a product $\mathcal{H}\times M'$ where $\mathcal{H}$ is a Hilbert space and $M'$ is a Riemannian symmetric space with nonpositive curvature operator and no  Euclidean local de Rham factor. The simply-connectedness allows us to avoid Riemannian symmetric spaces with vanishing sectional cur
like flat torus.
\end{rema}

The \textit{rank} of a metric space is the supremum of dimensions of Euclidean spaces isometrically embedded. The paper \cite{MR3044451} was focused on some irreducible infinite dimensional Riemannian symmetric spaces of nonpositive sectional curvature with finite rank. For brevity, the following notation was used in \cite{MR3044451} : $X_p(\K)$ ($p\in\N$) denotes the symmetric space $O^2(p,\infty)/O^2(p)\times O^2(\infty)$, $U^2(p,\infty)/U^2(p)\times U^2(\infty)$ or $Sp^2(p,\infty)/Sp^2(p)\times Sp^2(\infty)$ depending on wether $\K$ is the field of real, complex or quaternionic numbers. Actually, these spaces are the only irreducible ones to have infinite dimension and finite rank.
\begin{coro}\label{finite rank} Let $(M,g)$ be a separable Riemannian symmetric space with nonpositive curvature operator and no Euclidean local de Rham factor. The rank of $M$ is equal to its telescopic dimension. Moreover, if it is finite then 
\[M\simeq\prod_{i=1}^kM_i\]
where $M_i$ is an irreducible finite dimensional Riemannian symmetric space of noncompact type or is homothetic to some $X_p(\K)$.
\end{coro}
The telescopic dimension of a CAT(0) space is a notion of dimension at large scale introduced in \cite{MR2558883}.\\

We conclude this section with an example of  a space which is symmetric and has nonpositive curvature but which is not a Riemannian symmetric space. This is a purely infinite dimensional phenomenon. Let $(X,d)$ be a metric space. We say that $X$ is a \textit{CAT(0) symmetric space} if it is a complete CAT(0) space such that for any point $x\in X$, there exists an involutive isometry $\sigma_x$ with unique fixed point $x$. Observe that this condition implies that $x$ is the midpoint of $y$ and $\sigma_x(y)$ for any $y\in X$. In finite dimension,  \cite[Theorem 1.1]{MR2574740} implies that any proper CAT(0) symmetric space is the product of a Euclidean space and a Riemannian symmetric space of noncompact type (and finite dimension). This theorem uses the solution to Hilbert’s fifth problem and local compactness is crucial.

Let  $\mathbb{H}$ be the hyperbolic plane with sectional curvature -1  and let $o$ be a point in $\mathbb{H}$. We set L$^2([0,1],\mathbb{H})$ to be the space of measurable maps $f\colon[0,1]\to\mathbb{H}$ such that $\int d(f(t),o)^2\mathrm{d}t\langle\infty$. This space is a CAT(0) symmetric space but not a Riemannian manifold, see Section \ref{counterexample}.
\subsection{Nonnegative curvature}
In the case of nonnegative curvature,  some more technicalities appear. The first one is the lack of automatic simply-connectedness and the second one is the fact that the exponential map is not necessarily a diffeomorphism. Under the assumption of simply-connectedness, we obtain the following theorem. 
\begin{theorem}\label{nonnegativebis}Let $(M,g)$ be a simply-connected separable Riemannian symmetric space with nonnegative curvature operator then $(M,g)$ is isometric to a Hilbertian product 
\[M\simeq\mathcal{H}\times{\prod_i}^2 M_i\]
where $\mathcal{H}$ is a Hilbert space and each $M_i$ is a simply-connected irreducible Riemannian symmetric space. Each $M_i$ can be a finite dimensional Riemannian symmetric space of compact type or is homothetic to an element of the following list.

\[\widetilde{U^2(\infty)/SO^2(\infty)},\quad \widetilde{U^{ 2}(\infty)/Sp^2(\infty)},\quad U^2(p+\infty)/U^2(p)\times U^2(\infty),
\quad SO^2(p+\infty)/SO^2(p)\times SO^2(\infty)\]
\[SO^{ 2}(\infty)/U^2(\infty),\quad Sp^2(\infty)/U^2(\infty),\quad Sp^2(p+\infty)/Sp^2(p)\times Sp^2(\infty), \]
\[\widetilde{U^2(\infty)},\quad \widetilde{SO^2(\infty)},\quad Sp^2(\infty)\]
where $p\in\N\cup\{\infty\}$.
\end{theorem}
\subsection{Comments} 
W. Kaup obtained a classification of Hermitian symmetric spaces in \cite{MR639580,MR690007}.  His work uses the so-called Jordan-Hilbert algebras (Jordan algebras with a compatible structure of Hilbert space and an adjoint map $X\mapsto X^*$). His technics seem difficult to adapt to the real case. The paper \cite{MR2526791} shows a description in terms of L$^*$-algebras of the irreducible Hermitian symmetric spaces. The approach of symmetric space of W. Kaup is closer to the one of O. Loos  than the one of \'E. Cartan. Generalizations of Loos' approach to symmetric spaces can be found in \cite{MR1809879} and in \cite{Klotz:2009qy,MR1950888,MR776786} for Banach symmetric spaces.\\

\textbf{Acknowledgments.} The author thanks Pierre de la Harpe for useful comments on a previous version of this article, Wolfgang Bertram for pointing out interesting and relevant references and Julien Maubon for illuminating discussions. Moreover, the referee did a great job which helped the author to correct some mistakes and improve readability. It is a pleasure to thank him.

\section{L$^*$-algebras}\label{L$^*$-algebras}
\subsection{Definitions}\label{subb}
\begin{defi}A L$^*$-\textit{algebra} is a Lie algebra with a structure of (complex or real) Hilbert space such that there is a map $x\mapsto x^*$ satisfying, for all $x,y,z$,  the equation
\begin{equation}\label{L^*}
\langle[x,y],z\rangle=\langle y,[x^*,z]\rangle.
\end{equation}
\end{defi}
An \emph{ideal} of a L*-algebra $L$ is an ideal of the underlying Lie algebra which is moreover closed and $*$-invariant. Observe that an ideal of L*-algebra is a L*-algebra on its own. A L$^*$-algebra $L$ is \textit{semisimple} if $\overline{[L,L]}=L$ and it is \textit{simple} if it has no nontrivial ideal. A L$^*$-algebra is \textit{compact} if it is semisimple and $x^*=-x$ for all $x$. A L$^*$-algebra is \textit{noncompact} if it is semisimple and has no nontrivial compact ideal. An \textit{isomorphism} between L$^*$-algebra is an isomorphism of Lie algebras that is also an isometry and intertwines the involutions.
\begin{exem}Let $\mathcal{H}$ be a separable Hilbert space over $\K=\R,\C$ or the field of quaternions. The Lie algebra of Hilbert-Schmidt operators endowed with the involution given by the adjoint and the Hilbert structure coming from the Hilbert-Schmidt scalar product is a L$^*$-algebra, which we denote by $\mathfrak{gl}^2_\infty(\K)$. A choice of a Hilbert base for $\mathcal{H}$ provides embeddings of the algebras of operators $\mathfrak{gl}_n(\K)$ on $\K^n$, into  $\mathfrak{gl}^2_\infty(\K)$ such that their increasing union is dense. The other examples of separable simple L$^*$-algebras are constructed in a similar way.
\end{exem}

\begin{exem}\label{fintitedimension}Let $\mathfrak{g}$ be a semisimple real Lie algebra of finite dimension. Let $\mathfrak{g}=\mathfrak{k}\oplus\mathfrak{p}$ be a Cartan decomposition of $\mathfrak{g}$. The Killing form $B$ of $\mathfrak{g}$ is negative definite on $\mathfrak{k}$ and positive definite on $\mathfrak{p}$. Moreover for any $X,Y,Z$, we have $B([X,Y],Z)=-B(Y,[X,Z])$. Hence, if we define $(K+P)^*=-K+P$ (with $K\in\mathfrak{k}$ and $P\in\mathfrak{p}$) and $\langle X,Y\rangle=B(X,Y^*)$  then $(\mathfrak{g},\langle\ ,\ \rangle)$ is a L$^*$-algebra. Actually, the map $X\mapsto X^*$ is just the opposite of the Cartan involution.

\end{exem}
 
 Any separable L$^*$-algebra can be written as the direct sum of an Abelian ideal and a Hilbertian sum (as described below) of simple ideals. Moreover, simple L$^*$-algebras have been classified in the complex and real cases (see \cite{MR0117575,MR0133408,MR0347920,MR0282218,MR0325721}). The simple separable infinite dimensional real L$^*$-algebras which are compact and noncompact are recalled respectively in Table \ref{listec} and \ref{liste}.

\begin{table}[h]
\begin{center}
\begin{tabular}{|c|l|}
\hline
Type& Algebra\\
\hline
A&$\mathfrak{u}^2(\infty)$\\
BD&$\mathfrak{o}^2(\infty)$\\
C&$\mathfrak{sp}^2(\infty)$\\
\hline
\end{tabular}\caption{List of simple compact L$^*$-algebras}\label{listec}
\end{center}
\end{table}

\begin{table}
\begin{center}
\begin{tabular}{|c|l|}
\hline
Type& Algebra\\
\hline
A I&$\mathfrak{gl}_\infty^2(\R)$\\
A II&$\mathfrak{u}_\infty^{*\ 2}(\C)$\\
A III&$\mathfrak{u}^2(p,\infty),\ p\in\N^*\cup\{\infty\}$\\
BD I&$\mathfrak{o}^2(p,\infty),\ p\in\N^*\cup\{\infty\}$\\
BD III&$\mathfrak{o}^{*\ 2}(\infty)$\\
C  I&$\mathfrak{sp}_\infty^2(\R)$\\
C II&$\mathfrak{sp}^2(p,\infty),\ p\in\N^*\cup\{\infty\}$\\
\hline
A&$\mathfrak{gl}_\infty^2(\C)$\\
BD&$\mathfrak{o}_\infty^2(\C)$\\
C&$\mathfrak{sp}_\infty^2(\C)$\\
\hline
\end{tabular}\caption{List of simple noncompact L$^*$-algebras}\label{liste}
\end{center}
\end{table}

The last three algebras in Table \ref{liste} are moreover complex simple L$^*$-algebras. The notations used here are maybe not standard but we hope the correspondence with notations used in \cite{MR0282218} or \cite{MR0325721} is transparent. They are chosen to be brief and close to the ones used in finite dimension \cite[Tables IV and V, X.6]{MR1834454}. We refer to the previous references for a description of these algebras.

Let $\{\mathcal{H}_i\}$ be a countable family of separable (real, complex or quaternionic) Hilbert spaces. The \textit{Hilbertian sum} of this family, which we will denote by $\oplus^2\mathcal{H}_i$, is the set of sequences $v=(v_i)$ such that $\sum_i||v_i||^2$ is finite (see \cite[V.2.1]{MR910295}). Endowed with  the inner product $\langle u,v\rangle=\sum_i\langle u_i,v_i\rangle$, the space $\oplus^2\mathcal{H}_i$ is also a separable Hilbert space.


\begin{prop}Let $(L_i)$ be a countable family of semisimple $L^*$-algebras such that there exists $C\geq0$ with $||\mathrm{ad}(x)||\leq C||x||$ for all $i$ and all $x\in L_i$. For $x=(x_i),y=(y_i)\in\oplus^2L_i$, set $[x,y]=([x_i,y_i])$ and $x^*=(x_i^*)$. Endowed with this structure, the Hilbertian sum $\oplus^2L_i$ is a L$^*$-algebra.
\end{prop}
\begin{proof} Let $(x_i)\in \oplus^2L_i$ and $y=(y_i)\in\oplus^2L_i$ then $[x,y]=\sum[x_i,y_i]$ is an element of $\oplus^2L_i$ since $||[x,y]||^2\leq\sum C^2||x_i||^2||y_i||^2\leq C^2||x||^2||y||^2$. This also shows that ad$(x)$ is a linear bounded operator and the Lie bracket is also continuous. Continuity arguments show that $\oplus^2L_i$ is a Lie algebra and for all $x\in\oplus^2L_i$, ad$(x)^*=$ad$(x^*)$. Since $L_i$ is semisimple, the equation (\ref{L^*}) implies that $\langle u_i,v_i^*\rangle=\langle v_i,u_i^*\rangle$ for all $u_i\in L_i$ and $v_i\in [L_i,L_i]$ (see \cite[Preliminaries]{MR0117575}). Since $\overline{[L_i,L_i]}=L_i$, we have $||u_i^*||=||u_i||$ for any $u_i\in L_i$. Finally, $(x_i^*)\in\oplus^2 L_i$. 
\end{proof}
\begin{rema} In the preliminaries of \cite{MR0117575}, R. Schue wrote : ``The Hilbert space direct sum of $L^*$-algebras defines an $L^*$-algebra in the obvious way". Actually, the condition on the uniform bound of operators ad$(x)$ is necessary.
\end{rema}
%
%
%
\subsection{Orthogonal symmetric L*-algebras}
Orthogonal symmetric Lie algebras of finite dimension play an important role in the theory of finite dimensional Riemannian symmetric spaces. We give the following definition in the context of semisimple L$^*$-algebras. 
\begin{defi}\label{osla}An \textit{orthogonal symmetric} L$^*$-\textit{algebra} is a pair $(L,s)$ where
\begin{enumerate}[(i)]
\item $L$ is a real $L^*$-algebra, 
\item $s$ is an involutive isometric automorphism of the L$^*$-algebra $L$, 
\item For all $X\in L$ such that $s(X)=X$, $X^*=-X$.
\end{enumerate}
A symmetric orthogonal L$^*$-algebra $(L,s)$, is called \textit{irreducible} if it has no $s$-invariant ideal.

\end{defi}
In finite dimension, there is a duality between orthogonal symmetric Lie algebras of compact type and orthogonal symmetric Lie algebras of noncompact type (see, e.g., \cite[Section V.2]{MR1834454}). This duality extends to the context of L$^*$-algebras.

Let $(L,s)$ be a symmetric orthogonal L$^*$-algebra and let $\tilde{L}$ be its complexification as L$^*$-algebra \cite[§1.1]{MR0325721}. In particular, the extension of the map $X\mapsto X^*$ is conjugate linear. The automorphism $s$ extends linearly to a L$^*$-automorphism of $\tilde{L}$. Let $L=\mathfrak{k}\oplus\mathfrak{p}$ be the decomposition of $L$ into $+1$ and $-1$ eigenspaces of $s$. 
\begin{defi} The real L$^*$-algebra $L'= \mathfrak{k}\oplus i\mathfrak{p}$ endowed with the restriction $s'$ of $s$ on $L'$ is called the dual of $(L,s)$.
\end{defi}

\begin{lemm}\label{uninteresting}Let $(L,s)$ be a symmetric orthogonal L$^*$-algebra. Then :
\begin{enumerate}
\item The pair $(L',s')$ is an orthogonal symmetric L$^*$-algebra.
\item The pair $((L')',(s')')$ is isomorphic to $(L,s)$ as symmetric orthogonal L$^*$-algebra.
\item Assume that $L$ is a simple L$^*$-algebra. Then $L$ is compact if and only if $L'$ is noncompact.
\end{enumerate}
\end{lemm}

\begin{proof}The vector space $L'$ is a real L$^*$-subalgebra of $\tilde{L}$ which invariant under the extension of $s$ to $\tilde{L}$. Thus $s'$ is an involutive isometric automorphism of $L'$ and $L'=\mathfrak{k}\oplus i\mathfrak{p}$ is the decomposition of $L'$ into $+1$ and $-1$ eigenspaces of $s'$. In particular, for any $X\in\mathfrak{k}$, $X^*=-X$.

The isomorphism between $(L')'$ and $L$ comes from the identification of the decompositions $\tilde{L}=(\mathfrak{k}\oplus\mathfrak{p})\oplus i(\mathfrak{k}\oplus\mathfrak{p})$ and $\tilde{L'}=(\mathfrak{k}\oplus i\mathfrak{p})\oplus i(\mathfrak{k}\oplus i\mathfrak{p})$ .

If $L$ is simple and compact then for any $X\in \mathfrak{p}$, $(iX)^*=-iX^*=iX$ and thus $L'$ is noncompact. Conversely, assume that $L$ is simple and noncompact. Let $L=\mathfrak{k'}\oplus\mathfrak{p'}$ be the decomposition of $L$ into -1 and +1 eigenspaces of $*$ and let $L=\mathfrak{k}\oplus\mathfrak{p}$ be its decomposition into +1 and -1 eigenspaces of $s$. By assumption, we know that $\mathfrak{k}\subseteq\mathfrak{k}'$ and thus $\mathfrak{p}'\subseteq\mathfrak{p}$. Observe that $[\mathfrak{p},\mathfrak{p}]\subseteq\mathfrak{k}$ and that $\overline{[\mathfrak{p}',\mathfrak{p}']}=\mathfrak{k}'$.  Thus, $\mathfrak{k}'=\overline{[\mathfrak{p}',\mathfrak{p}']}\subseteq\overline{[\mathfrak{p},\mathfrak{p}]}\subseteq\mathfrak{k}$. This shows that $X*=-s(X)$ for any $X\in L$. Now, for $X$ in $\mathfrak{p}$, $(iX)^*=-iX^*=is(X)=-iX$ and $L'$ is compact.
\end{proof}

\begin{prop}\label{compact}Let $(L,s)$ be a separable compact orthogonal symmetric L$^*$-algebra. Let $L=\oplus^2 L_i$ be the decomposition of $L$ into simple ideals then $s$ permutes the $L_i$'s. The algebra $L$ is the Hilbertian sum of irreducible orthogonal symmetric L$^*$-algebras $I_k$. Each  $I_k$ is equal to some $s$-invariant simple ideal or $I_k=L_i\oplus L_j$ with $s(L_i)=L_j$ for some $L_i$ and $L_j$.

If $I_k=L_i\oplus L_j$ with $s(L_i)=L_j$ then $L_i$ is isomorphic to $L_j$ which is isomorphic to $\mathfrak{o}^2(\infty),\mathfrak{u}^2(\infty)$ or $\mathfrak{sp}^2(\infty)$. The decomposition $I_k=\mathfrak{k}\oplus\mathfrak{p}$ into $+1$ and $-1$ eigenspaces of $s$ is given by $\mathfrak{k}=\{X+s(X);\ X\in L_i\}$ and $\mathfrak{p}=\{X-s(X);\ X\in L_i\}$.

Assume $L_i$ is $s$-invariant. If we decompose $L_i=\mathfrak{k}\oplus\mathfrak{p}$ into $+1$ and $-1$ eigenspaces of $s$ then $L_i$ is isomorphic to one orthogonal symmetric L$^*$-\textit{algebra} of Table \ref{compactorthogonal}.
\end{prop}
\begin{table}
\begin{center}
\begin{tabular}{|c|c|l|}
\hline
Type&L$^*$-algebra&$\mathfrak{k}$\\
\hline
AI&$\mathfrak{u}^2(\infty)$&$\mathfrak{o}^2(\infty)$\\
AII&$\mathfrak{u}^2(\infty)$&$\mathfrak{sp}^2(\infty)$\\
AIII&$\mathfrak{u}^2(p+\infty)$&$\mathfrak{u}^2(p)\times\mathfrak{u}^2(\infty),\ p\in\N^*\cup\{\infty\}$\\
BDI&$\mathfrak{o}^2(p+\infty)$&$\mathfrak{o}^2(p)\times\mathfrak{o}^2(\infty),\ p\in\N^*\cup\{\infty\}$\\
BDIII&$\mathfrak{o}^2(\infty)$&$\mathfrak{u}^2(\infty)$\\
CI&$\mathfrak{sp}^2(\infty)$&$\mathfrak{u}^2(\infty)$\\
CII&$\mathfrak{sp}^2(p+\infty)$&$\mathfrak{sp}^2(p)\times\mathfrak{sp}^2(\infty),\ p\in\N^*\cup\{\infty\}$\\
\hline
\end{tabular}
\end{center}
\caption{List of compact simple orthogonal symmetric L$^*$-algebras.}\label{compactorthogonal}
\end{table}
\begin{rema}The description of simple compact orthogonal symmetric L$^*$-algebras in Table  \ref{compactorthogonal} has the advantage to be brief but it is not explicit. The subalgebra $\mathfrak{k}$ is given up to isomorphism but the embedding in $L_i$ and the involution are not given. An explicit description can be obtained in the proof of Proposition \ref{compact}, that is obtained as the dual of some noncompact simple L$^*$-algebra.

\end{rema}
\begin{proof}[Proof of Proposition \ref{compact}] Since $s$ is L$^*$-automorphism, the image of a simple ideal is also a simple ideal. The decomposition $L=\oplus^2 L_i$ is unique up to permutation. Therefore, for any $i$ there is $j$ such that $s(L_i)=L_j$.

Now is suffices to understand involutive L$^*$-automorphisms of compact simple L$^*$-algebras. Let $L_0$ be a  compact simple L$^*$-algebra with an involutive L$^*$-automorphism $s$. We decompose $L_0=\mathfrak{k}\oplus\mathfrak{p}$  into $\pm1$ eigenspaces of $s$. Let $\tilde{L}$ be the complexification of $L_0$. Since $L_0$ is compact, $L_0$ has no complex structure and thus (\cite[Theorem 1.3.1]{MR0325721}) $\tilde{L}$ is simple. Let $L$ be the real form of $\tilde{L}$ associated to $s$ (extended to $\tilde{L}$) (see \textit{loc. cit}). Since $L_0$ is compact, we know that $L=\mathfrak{k}\oplus i\mathfrak{p}$, that is the dual of $L_0$. The L$^*$-algebra $L$ is a noncompact simple L$^*$-algebra of and thus is one of those described in section \ref{subb} or more precisely in \cite[Section 5]{MR0325721}. Thanks to Lemma \ref{uninteresting}, $L_0$ is the dual of $L$ with its unique possible structure of orthogonal symmetric L$^*$-algebra.
\end{proof}

\section{Construction of a L$^*$-algebra}
\subsection{Riemannian symmetric spaces}
A \textit{Riemannian manifold} is a pair $(M,g)$ such that $M$ is a connected smooth manifold modeled on a real  Hilbert space and $g$ is a smooth Riemannian metric on $M$. Our standard reference for these manifolds  is  \cite{MR1666820} and in particular, we will adopt the same convention for the sign of the Riemann 4-tensor, which is also the sign used in \cite{MR1834454} for example, but is opposite to the one used in \cite{MR1330918}. With this convention, for two orthogonal unitary vectors $u,v$ of a tangent space $T_pM$, the sectional curvature is $Sec(u,v)=-R(u,v,u,v)$ where $R$ is the Riemann 4-tensor. This convention will also explain the minus sign which appears in the definition of the curvature operator.

\begin{defi}A \textit{Riemannian symmetric space} is a Riemannian manifold such that at each point $p\in M$, there is an isometry, $\sigma_p$ which leaves $p$ fixed and satisfies $d_p\sigma_p=-$Id.
\end{defi}

\begin{rema}The definition of symmetric spaces given in \cite[XIII,\S5]{MR1666820} is not the same as ours since Lang assumes that the exponential map is everywhere surjective. Neeb observed \cite[Remark 3.8]{MR1950888} that this additional property is unnecessary to use results of  \cite[XIII]{MR1666820}, on which we rely.
\end{rema}

We collect some remarks about metric completeness and geodesic completeness. In finite dimension, these two notions of completeness are equivalent thanks to Hopf-Rinow theorem. Moreover, in finite dimension, any of this two conditions implies the existence of a path of minimal length between two points. In general, a Riemannian manifold which is metrically complete is also geodesically complete but the converse is false (see\cite{MR0400283}).  Furthermore, J.H. McAlpin \cite{MR2614999} constructed a metrically complete Riemannian manifold  such that there are two points which are not joined by a path of minimal length (see \cite[Remark p.226]{MR1666820}).

\begin{lemm}\label{homogeneous} The isometry group of a Riemannian symmetric space $M$ acts transitively on $M$.
\end{lemm}

\begin{proof}Let  $x,y\in M$. There are points $x_0=x,x_1,\dots,x_n=y$ such that $x_{i-1}$ and $x_{i}$ are joined by a geodesic segment. Let $m_i$ be the midpoint of that segment. Now the isometry $\sigma_{m_{n}}\circ\dots\circ\sigma_{m_1}$ sends $x$ to $y$.
\end{proof}

In the case of a Riemannian symmetric space, metric completeness is a consequence of homogeneity and the existence of a closed ball that is complete and geodesic completeness is  proved in \cite[Proposition XIII.5.2]{MR1666820}.

If the sectional curvature is nonpositive then metric completeness is equivalent to geodesic completeness \cite[Corollary IX.3.9]{MR1666820}. This a consequence of a version of Cartan-Hadamard theorem due to J.H. McAlpin \cite{MR2614999} (see also \cite{MR1950888} for Banach manifolds). Since any Riemannian symmetric space is geodesically complete, this version of Cartan-Hadamard theorem \cite[Theorem IX.3.8]{MR1666820} implies also that the exponential map at any point is surjective.

In finite dimension, the condition of existence of a  local symmetry is equivalent to the parallelism of the Riemann tensor. The  same holds in infinite dimension.

\begin{defi} A Riemannian manifold $M$ is said to be locally symmetric if for any $p\in M$, there exists a ball $B$ around $p$ and an isometry $\sigma_p$ of this ball such that $d_p\sigma_p=-$Id.
\end{defi}

\begin{prop} A Riemannian manifold $M$ is locally symmetric if and only if $\nabla R=0$.
\end{prop}

\begin{proof}The proof of this fact in finite dimension (see e.g. \cite[IV.1]{MR1834454}) works as well in infinite dimension. The fact that a Riemannian locally symmetric space has parallel Riemann tensor is \cite[Proposition XIII.6.2]{MR1666820} and the converse relies on Cartan's theorem.
\end{proof}
Let us introduce some notations before stating Cartan's theorem. Let $(M,g)$ and $(M',g')$ be Riemannian manifolds modelled on the same Hilbert space. If $c$ is a geodesic curve $c\colon[a,b]\to M$ we denote by $\dot{c}(t)$ the tangent vector at $c(t)$ and $P_{a,c}^b$ the parallel transport along $c$. Parallelism of the Riemann tensor can be expressed with the following relation \cite[XIII, \S6]{MR1666820}
\begin{equation}\label{eqpara}
P_{a,c}^b\circ R_{\dot{c}(a)}=R_{\dot{c}(b)}\circ P_{a,c}^b.
\end{equation}
Let $p\in M$, $p'\in M'$ and $r>0$ be such that $B(p,r)$ and $B(p',r)$ are normal balls. Let $i_p\colon T_pM\to T_{p'}M'$ be an isometry.
We define $\Phi=\exp_{p'}\circ i_p\circ\exp_p^{-1}\colon B(p,r)\to B(p',r)$. Now, let $c$ be a radial geodesic with unit speed starting at $p$ and let $c'$ be its image by $\Phi$. For $0\leq t<r$ we set $i_t=P_{0,c'}^t \circ i_p \circ P_{t,c}^0$.

\begin{theorem}[Cartan's theorem {\cite[Theorem 1.12.8]{MR1330918}}]\label{cartantheorem}Assume that for all radial geodesics $c(t)$ and their images $c'(t)=\Phi\circ c(t)$ we have
\[i_t\circ R_{\dot{c}(t)}=R'_{\dot{c}'(t)}\circ i_t.\]
Then $\Phi$ is an isometry.
\end{theorem}

\begin{prop}\label{global} Let $M,M'$ be Riemannian symmetric spaces. Assume $M$ is simply-connected. Any local isometry from an open set of $M$ to $M'$ can be uniquely extended to an isometric covering map from $M$ to $M'$. If moreover $M'$ is simply-connected as well then this covering map is actually an isometry.\end{prop}

\begin{proof}  Recall that a ball $B(p,r)$ in a Riemaniann manifold $M$ is normal if the exponential map at $p$ realizes a diffeomorphism from the ball of radius $r$ in $T_p M$ onto $B(p,r)$. Since $M, M'$ are symmetric and thus homogeneous (Lemma \ref{homogeneous}), there is $r_0>0$ such that for any $x\in M$ and $x'\in M'$, $B(x,r_0)$ and $B(x',r_0)$ are normal balls. Let $x\in M$, $x'\in M'$ and $r>0$ such that $B(x,r)$ and $B(x',r)$ are isometric normal balls. Let us call $\phi$ this isometry. We aim to extend $\phi$ to $M$. Let $y\in M$ and $\gamma\colon[0,1]\to M$ be a continuous path such that $\gamma(0)=x$ and $\gamma(1)=y$. Choose a increasing sequence $t_0=0,t_1,\dots,t_n=1$ such that $\cup B(x_i,r/2)$ contains $\gamma([0,1])$ and $d(x_i,x_{i+1})<r$ where $x_i=\gamma(t_i)$. We show that $\phi$ is extendable along $\gamma$. See \cite[I.\S11]{MR1834454} for details about extendable isometries in finite dimension which works in infinite dimension as well. Let us denote by $\phi_i$ the restriction of $\phi$ on $B(x,r/2)$. Assume $\phi_i$ has been defined on $B(x_i,r/2)$ being an immediate continuation of $\phi_{i-1}$, it has at most one isometric extension on $B(x_i,r)$, which is given by $\exp_{\phi(x_i)}\circ d_{x_i}\phi\circ\exp_{x_i}^{-1}$. The fact that this extension is an isometry follows from Theorem \ref{cartantheorem} and the following computation based on Relation \eqref{eqpara}.
\begin{align*}
i_t\circ R_{\dot{c}(t)}&=P_{0,c'}^t \circ i_p \circ P_{t,c}^0\circ R_{\dot{c}(t)}\\
&=P_{0,c'}^t \circ i_p \circ R_{\dot{c}(0)}\circ P_{t,c}^0\\
&=P_{0,c'}^t \circ R_{\dot{c}'(0)}\circ i_p\circ P_{t,c}^0\\
&=R_{\dot{c}'(t)}\circ i_t\ .
\end{align*}
In particular, this extension of $\phi_i$ is well defined on a neighborhood of $x_{i+1}$ and the previous formula mutatis mutandis allows us to define an immediate continuation $\phi_{i+1}$ of $\phi_i$ on $B(x_{i+1},r/2)$.

A small continuous deformation of $\gamma$ remains in $\cup B(x_i,r/2)$ and thus the continuation of $\phi$ along such a small deformation gives same value to $y$ with same differential at $y$. This is the so-called \emph{monodromy theorem for isometries} \cite{MR0052177}. Now, since $M$ is simply-connected, one can extend $\phi$ to an isometric map $M\to M'$. Observe that for any $y'\in \phi(M)$ and $y\in M$ such that $\phi(y)=y'$ then $B(y,r)$ is isometricaly map onto $B(y',r)$ by $\phi$ by construction. In particular,  $B(y',r)\subset \phi(M)$. This shows that $\phi(M)$ is clopen and since $M'$ is connected, $\phi(M)=M'$.  The last statement of the proposition follows from the universal property of the universal cover.
\end{proof}

\begin{prop}\label{universal}The universal cover of a Riemannian symmetric space is a Riemannian symmetric space.
\end{prop}

\begin{proof}Let $M$ be a Riemannian symmetric space and let $\widetilde{M}$ be its universal cover. We endow $\widetilde{M}$ with the Riemannian structure coming from $M$. Since $M$ is homogeneous under this action of its isometry group, there exists $\varepsilon>0$ such that $B(p,\varepsilon)$ is normal neighborhood for any $p\in M$ and for any $p\in\widetilde{M}$, the projection $\pi\colon\widetilde{M}\to M$ induces an isometry  from $B(p,\varepsilon)$ to $B(\pi(p),\varepsilon)$ for any $p\in\widetilde{M}$.

Choose $p\in\widetilde{M}$ there exists an isometry $\sigma_0$ of $B(p,\varepsilon)$ fixing $p$ such that $d_p\sigma_0=-id$. We want to extent $\sigma_0$ to $\widetilde{M}$. Since $\widetilde{M}$ is simply-connected, it suffices to prove that $\sigma_0$ can extended along any continuous path starting at $p$, $\gamma\colon[0,1]\to\widetilde{M}$. By compactness of $\gamma([0,1])$, we can choose points $p_0=p,p_1,\dots,p_n$ on $\gamma([0,1])$ such that $\gamma([0,1])\subset\cup_i B(p_i,\varepsilon)$ and $d(p_i,p_{i+1})\langle\varepsilon/2$ for any $0\leq i\leq n$. Assume there is an isometric extension $\sigma_i\colon\cup_{j=0}^i B(p_j,\varepsilon)$ of $\sigma_0$ then the restriction $\sigma_i$ to $B(p_{i+1},\varepsilon/2)$ is an isometry which can be extended to $B(p_{i+1},\varepsilon)$ thanks to Theorem \ref{cartantheorem}. This isometry coincides with $\sigma_i$ on $\cup_{j=0}^n B(p_j,\varepsilon)\cap B(p_{i+1},\varepsilon)$ and  thus there is an extension $\sigma_{i+1}$ of $\sigma_0$ on $\cup_{j=0}^{i+1} B(p_j,\varepsilon)$.

The monodromy theorem for isometries and simply-connectedness show that there is a well-defined isometry $\sigma\colon\widetilde{M}\to\widetilde{M}$ fixing $p$ and satisfying $d_p\sigma=-id$. \end{proof} 

\subsection{Reminiscence of a Killing form}\label{construction}
For the remainder of this section $(M,g)$ will be a simply-connected Riemannian symmetric space. A \textit{Killing field} on $M$ is a smooth vector field such that its flow is realized by isometries (metric Killing vector field in the terms of \cite{MR1666820}). Let $\mathfrak{g}$ be the Lie algebra of Killing fields of $M$ and let $p$ be a point in $M$. The Lie algebra $\mathfrak{g}$  has a direct decomposition $\mathfrak{g}=\mathfrak{q}\oplus\mathfrak{p}$ where $\mathfrak{p}$ identifies with $T_pM$ under the map $X\mapsto X(p)$ and $\mathfrak{q}$ is the kernel of this map (see \cite[Theorem XIII.5.8]{MR1666820}). Moreover, we have the  following relations (see \cite[Theorem XIII.4.4]{MR1666820})
\begin{align*}
[\mathfrak{q},\mathfrak{q}]&\subseteq \mathfrak{q}\\
[\mathfrak{p},\mathfrak{p}]&\subseteq \mathfrak{q}\\
[\mathfrak{q},\mathfrak{p}]&\subseteq \mathfrak{p}.
\end{align*}
The Riemann 4-tensor has a particular expression (see \cite[Theorem XIII.4.6]{MR1666820}) in this case : for any $X,Y,Z,T\in T_pM\simeq\mathfrak{p}$, \begin{equation}\label{tensor}
R(X,Y,Z,T)=g([Z,[X,Y]],T).
\end{equation}

\begin{rema}
In the particular case of a finite dimensional irreducible symmetric space, the metric on the tangent space is a multiple of the Killing form $B$ of the group of isometries and thus
\begin{equation}\label{Killing}
R(X,Y,Z,T)=\lambda B([X,Y],[Z,T]),\ \lambda\in\mathbb{R}^*.
\end{equation}
\end{rema}

 In finite or infinite dimension, the symmetries of $R$ allows us to define a symmetric bilinear form on the alternating algebraic tensor product $\Wedge\mathfrak{p}$ by \[(X\wedge Y,Z\wedge T)=R(X,Y,Z,T).\] The space $\Wedge\mathfrak{p}$ has also a structure of preHilbert space defined by 
\[ \langle X\wedge Y,Z\wedge T\rangle_g=\det\left[\begin{array}{cc}
g(X,Z)&g(X,T)\\
g(Y,Z)&g(Y,T)
\end{array}\right].\]
With these notations, the sectional curvature of two vectors $X,Y\in T_pM$ is 
\[\mathrm{Sec}(X,Y)=-\frac{(X\wedge Y,X\wedge Y)}{\langle X\wedge Y,X\wedge Y\rangle_g}.\]
The vector space $\Wedge\mathfrak{p}$ can be naturally identified with the space of finite rank and skew-symmetric operators of $\mathfrak{p}$. The tensor $X\wedge Y=X\otimes Y-Y\otimes X$ is identified with the operator $Z\mapsto \langle X,Z\rangle Y-\langle Y,Z\rangle X$. This identification is actually an isometry when the space of finite rank operators is seen as a subspace of Hilbert-Schmidt operators with the Hilbert-Schmidt norm (up to a factor $\sqrt{2}$). For a bounded operator $A$ and a finite rank operator $B$ on $\mathfrak{p}$, we define \[\langle A,B\rangle_g=\trace(^tAB).\] For example, if $A$ is a bounded operator and $X,Y\in\mathfrak{p}$ then $\langle A,X\wedge Y\rangle_g=g(AX,Y)-g(X,AY)$. 

In finite dimension (see, e.g., \cite[Section 2.2]{MR2243772} or  \cite[§4]{MR0454884}), $\langle,\rangle_g$ is simply the Hilbert-Schmidt scalar product on L($\mathfrak{p}$) (where L$(\mathfrak{p})$ is the space of linear bounded operators on $\mathfrak{p}$) and thus there is a symmetric  operator $C$ of $\Wedge\mathfrak{p}$  such that \[(X\wedge Y, Z\wedge T)=-\langle C(X\wedge Y),Z\wedge T\rangle_g\] for $X,Y,Z,T\in\mathfrak{p}$. This operator is called the \textit{curvature operator} of $M$. We generalize this construction in infinite dimension.

\begin{defi}\label{curvop}The curvature operator of $M$ is   the linear operator $C\colon\Wedge\mathfrak{p}\to$L$(\mathfrak{p})$ 
such that $(X\wedge Y,Z\wedge T)=-\langle C(X\wedge Y),Z\wedge T\rangle_g$. 
\end{defi}

Actually, $C(X\wedge Y)$ is skew-symmetric and thanks to equation (\ref{tensor}), we know that $C(X\wedge Y)Z=1/2[Z,[X,Y]]$.

We say that the curvature operator is \textit{nonpositive} (respectively \textit{nonnegative}) if for any $U\in\Wedge\mathfrak{p}$, $\langle C(U),U\rangle_g\leq0$ (respectively $\langle C(U),U\rangle_g\geq0$). Observe that $C$ is nonpositive (respectively nonnegative) if for any families $(X_i)_{i=1\dots n}$, $(Y_i)_{i=1\dots n}$, \[\sum_{i,j=1}^nR(X_i,Y_i,X_j,Y_j)\geq 0\] (respectively $\sum_{i,j}R(X_i,Y_i,X_j,Y_j)\leq 0$).

Now we assume that $(M,g)$ is a Riemannian symmetric space of  fixed-sign curvature operator. For brevity, we will write $M$ is NPCO (resp. NNCO) if $M$ has nonpositive curvature operator (resp. nonnegative curvature operator). We want to endow $[\mathfrak{p},\mathfrak{p}]$ with a structure of preHilbert space. For $U=\sum_i[X_i,Y_i]$ and $V=\sum_j[Z_j,T_j]$, we define $\langle U,V\rangle=-\sum_jg([U,Z_j],T_j)$ if $M$ is NPCO and $\langle U,V\rangle=\sum_jg([U,Z_j],T_j)$ if $M$ is NNCO. For example, if $M$ is NPCO
\[\langle U,V\rangle=\sum_{i,j}R(X_i,Y_i,Z_j,T_j)=\sum_{i,j}(X_i\wedge Y_i,Z_j\wedge T_j).\]

\begin{lemm} The bilinear form $\langle\ ,\ \rangle$ is a scalar product on $[\mathfrak{p},\mathfrak{p}]\oplus\mathfrak{p}$ such that $\mathfrak{p}$ and $[\mathfrak{p},\mathfrak{p}]$ are orthogonal and its restriction to $\mathfrak{p}$ is $g$. 
\end{lemm}

\begin{proof}The symmetries of the Riemann tensor imply that $\langle\ ,\ \rangle$ is a symmetric bilinear form and the hypothesis on the curvature operator implies this form is nonnegative in both cases. The relation $R(X,Y,Z,T)=R(Z,T,X,Y)$ for $X,Y,Z,T\in\mathfrak{p}$ implies  for any $U\in[\mathfrak{p},\mathfrak{p}]$ that\begin{equation}\label{NPCO}
g([X,U],Y)=\langle U,[X,Y]\rangle
\end{equation} 
if $M$ is NPCO and 
\begin{equation}\label{NNCO}
g([X,U],Y)=-\langle U,[X,Y]\rangle
\end{equation}
if $M$ is NNCO. Moreover, the Cauchy-Schwarz inequality implies that if $\langle U,U\rangle=0$ then for any $X,Y\in\mathfrak{p}$, $g([U,X],Y)=\pm\langle U,[X,Y]\rangle=0$ and thus the Killing field $U$ is trivial.
\end{proof}
 We denote by $\mathfrak{k}$ the completion of $[\mathfrak{p},\mathfrak{p}]$ with respect to $\langle\ ,\ \rangle$, we extend $\langle\ ,\ \rangle$ on $\mathfrak{k}\oplus\mathfrak{p}$ and we denote by $||\ ||$ the associated norm. Thus, ($\mathfrak{k}\oplus\mathfrak{p}$, $\langle\ ,\ \rangle$) is a separable Hilbert space.
\begin{proof}[Proof of Theorem \ref{existence}]We show that the Lie algebra structure on $[\mathfrak{p},\mathfrak{p}]\oplus\mathfrak{p}$ extends to a L$^*$-algebra structure on $\mathfrak{k}\oplus\mathfrak{p}$. Since the Riemann 4-tensor is a bounded 4-linear form at each point, there exists a constant $\kappa$ such that $R(X,Y,Z,T)\leq\kappa||X||\, ||Y||\, ||Z||\, ||T||$ for any $X,Y,Z,T\in\mathfrak{p}$. Thus $||[X,Y]||\leq \sqrt{\kappa}||X||\, ||Y||$. If $U\in\mathfrak{k}$ and $X,Y\in\mathfrak{p}$ then $|\langle X,[U,Y]\rangle|=|\langle U,[X,Y]\rangle|\leq||U||\cdot||[X,Y]||$. The Lie bracket extends continuously to $\mathfrak{k}\times\mathfrak{p}$ and any $U\in\mathfrak{k}$ defines a bounded skew-symmetric operator ad$(U)\colon X\mapsto[U,X]$.

Moreover, Jacobi's identity for $U\in[\mathfrak{p},\mathfrak{p}]$ and $X,Y\in\mathfrak{p}$,
\[ [U,[X,Y]]=[[U,X],Y]+[X,[U,Y]],\]
shows that $[\mathfrak{p},\mathfrak{p}]$ is a subalgebra of the algebra of Killing fields. Observe that for $U\in\mathfrak{k}, t\in\R$, $\exp(t\operatorname{ad}(U))$ is an isometry of $\mathfrak{p}$ which preserves $R$. Thus, if $\Phi_t$ is defined as $\exp_p\circ\exp(t\operatorname{ad}(U))\circ\exp_p^{-1}$ on a ball around $p$ where $\exp_p$ is a diffeomorphism, then Theorem \ref{cartantheorem} shows that $\Phi_t$ is an isometry which can be extended in an isometry (also denoted $\Phi_t$) of $M$ thanks to Proposition \ref{global}. Thus, $t\mapsto\Phi_t$ is a smooth 1-parameter group of isometries fixing $p$ and the Killing field $U_0$ corresponding to this 1-parameter group satisfies $[U_0,X]=[U,X]$ for any $X\in\mathfrak{p}$ and we can identify $U$ with $U_0$. In particular, $\mathfrak{k}$ identifies with a subalgebra of $\mathfrak{q}$.

We now define the involution. For $U\in\mathfrak{k}$, we set $U^*=-U$ and for $X\in\mathfrak{p}$, we set $X^*=X$ if the curvature operator is nonpositive and $X^*=-X$ if the curvature operator is nonnegative. It remains to show that \begin{equation}\label{sum}\langle[X,Y],Z\rangle=\langle Y,[X^*,Z]\rangle\end{equation}
 for any $X,Y,Z\in\mathfrak{k}\oplus\mathfrak{p}$. Thanks to linearity and relations $[\mathfrak{k},\mathfrak{k}]\subseteq\mathfrak{k}$, $[\mathfrak{p},\mathfrak{p}]\subseteq\mathfrak{k}$, $[\mathfrak{k},\mathfrak{p}]\subseteq\mathfrak{p}$ and $\mathfrak{k}\bot\mathfrak{p}$, it suffices to show Equation (\ref{sum}) in the case $X\in\mathfrak{k}$, $Y,Z\in\mathfrak{p}$ and in the case $X,Y,Z\in\mathfrak{k}$. Suppose that $X\in\mathfrak{k}$, $Y,Z\in\mathfrak{p}$ then using Equations (\ref{NPCO}) and (\ref{NNCO}) we have
\[\langle[X,Y],Z\rangle=\pm\langle X,[Z,Y]\rangle=\mp\langle X,[Y,Z]\rangle=\mp\langle[X,Z],Y\rangle=\langle Y,[X^*,Z]\rangle.\]
For the case $X,Y,Z\in\mathfrak{k}$, thanks to continuity and linearity, we assume that $X=[X_1,X_2]$ for some $X_1,X_2\in\mathfrak{p}$. We treate only the case where $M$ is NPCO, the other case is similar.
\begin{align*}
\langle[X,Y],Z\rangle&=\langle[[X_1,X_2],Y],Z\rangle\\
&=-\langle[[Y,X_1],X_2]+[X_1,[Y,X_2]],Z\rangle\\
&=-\langle[Y,X_1],[Z,X_2]\rangle-\langle[Y,X_2],[X_1,Z]\rangle\\
&=-\langle Y,[[Z,X_2],X_1]+[[X_1,Z],X_2]\rangle\\
&=\langle Y,[Z,[X_1,X_2]]\rangle\\
&=-\langle Y,[X,Z]\rangle=\langle Y,[X^*,Z]\rangle.
\end{align*}

\end{proof}
\begin{defi}Let $(M,g)$ be a simply-connected Riemannian symmetric space with fixed-sign curvature operator and let $L=\mathfrak{k}\oplus\mathfrak{p}$ be a L$^*$-algebra associated to $M$ as in Theorem  \ref{existence}. We call merely $L$ the L$^*$-algebra associated to $M$. The Cartan involution of $L$ associated to this decomposition is the map $\theta\colon L\to L$ defined by $\theta(U+X)=U-X$ for $U\in\mathfrak{k}$ and $X\in\mathfrak{p}$. 
\end{defi}

\begin{lemm}The pair $(L,\theta)$ is an orthogonal symmetric  L$^*$-algebra.
\end{lemm}

\begin{proof}It is clear from the definition of $\theta$ that it is an involutive isometric map and that points (i) and (iii) of Definition \ref{osla} are satisfied. Relations $[\mathfrak{p},\mathfrak{p}]\subseteq\mathfrak{k}$, $[\mathfrak{k},\mathfrak{p}]\subseteq\mathfrak{p}$ and $[\mathfrak{k},\mathfrak{k}]\subseteq\mathfrak{k}$ show that $\theta$ is an automorphism of the Lie algebra $L$.
\end{proof}

\begin{prop}\label{uniqueness} Let $(M,g)$ be a simply-connected Riemannian symmetric space and let $L$, $L'$ be L$^*$-algebras with orthogonal decompositions $L=\mathfrak{k}\oplus\mathfrak{p}$ and $L'=\mathfrak{k}'\oplus\mathfrak{p}'$ satisfying (i) and (ii) of Theorem \ref{existence} then $L$ and $L'$ are isomorphic.
\end{prop}
\begin{proof}First, $\mathfrak{p}$ and $\mathfrak{p'}$ are isometric as Hilbert spaces and they generate isomorphic Lie algebras. Now, it suffices to observe that this isomorphism is also an isometry since the inner products are determined by their respective restrictions on  $\mathfrak{p}$ and $\mathfrak{p'}$.
\end{proof}
We state a little bit more precise theorem than Theorem \ref{localisomorphism}. %
\begin{theorem}\label{isometry}Let $(M,g)$ and $(M',g')$ be simply connected Riemannian symmetric spaces with fixed-sign curvature operator. Let $p\in M$ and $p'\in M'$and let $L,L'$ be the  two L$^*$-algebras with orthogonal decompositions $L=\mathfrak{k}\oplus\mathfrak{p}$ and $L'=\mathfrak{k}'\oplus\mathfrak{p}'$ associated to $M$ and $M'$ with respect to $p\in M$ and $p'\in M'$.

Assume there exists an isomorphism  of L$^*$-algebras between $L$ and $L'$ which intertwines the previous orthogonal decompositions. Then $M$ and $M'$ are isometric.\end{theorem}

The isometry will be provided by Theorem \ref{cartantheorem}. For any Riemannian manifold $N$ with Riemannian 4-tensor $R$, a point $q\in N$ and $X\in T_qN$, we denote by $R_X\colon T_qN\to T_qN$ the symmetric operator such that $R_X(Y)=R(X,Y)X=[X,[X,Y]]$ for any $Y\in T_qN$. In the symetric case, $R_X(Y)=[X,[X,Y]]$.

\begin{proof}[Proof of Theorem \ref{isometry}]Choose $r>0$ such that $B(p,r)$ and $B(p',r)$ are normal balls. Let $\varphi$ be an isomorphism between $L$ and $L'$ such that $\varphi(\mathfrak{k})=\mathfrak{k}'$ and $\varphi(\mathfrak{p})=\mathfrak{p}'$. We define $i_p\colon T_pM\to T_{p'}M'$ to be the restriction of $\varphi$ to $\mathfrak{p}$ identified with $T_pM$. The map $i_p$ is a linear isometry between Hilbert spaces. 

First, since $\varphi$ is a Lie algebra isomorphism and an isometry \[R'(\varphi(X),\varphi(Y),\varphi(Z),\varphi(T))=\langle[\varphi(Z),[\varphi(X),\varphi(Y)]],\varphi(T)\rangle=R(X,Y,Z,T)\] for any $X,Y,Z,T\in T_pM$. Parallelism of the Riemann tensor (Equation \eqref{eqpara}) implies that for any radial geodesic $c\colon [0,r]\to M$ with $c(0)=p$, $i_t\circ R_{\dot{c}(t)}=R_{\dot{c}'(t)}\circ i_t$ (see proof of Proposition \ref{global} for the computation) and the hypotheses of Cartan's theorem are now satisfied. This shows that $B(p,r)$ and $B(p',r)$ are isometric and thanks to Proposition \ref{global}, $M$ and $M'$ are isometric.
\end{proof}
The following proposition gives a natural condition which implies a decomposition as asked in Question \ref{question}.
\begin{prop}\label{union}Let $M$ be a Riemannian symmetric space. If there exists a dense increasing union of totally geodesic subspaces of finite dimension containing a point $p\in M$, then there is an orthogonal decomposition
\[T_pM=\mathfrak{p}_-\oplus\mathfrak{p}_0\oplus\mathfrak{p}_+\]
such that
\begin{itemize}
\item the subspaces $\mathfrak{p}_-$, $\mathfrak{p}_0$ and $\mathfrak{p}_+$ are commuting Lie triple systems of the Lie  algebra of Killing fields,
\item the restrictions of the curvature operator are nonnegative on $\mathfrak{p}_-$, trivial on $\mathfrak{p}_0$ and nonpositive on $\mathfrak{p}_+$.
\end{itemize}
\end{prop}
\begin{proof} Let $(M_n)$ be an increasing sequence of finite dimensional totally geodesic subspaces of $M$ such that their union is dense in $M$. Choose $p\in M_1$ and let $R^{M_n}$ be the Riemannian tensor of $M_n$ at $p$. Since $M_n$ is totally geodesic in $M$, for any $X,Y,Z,T\in T_pM_n$, $R^{M_n}(X,Y,Z,T)=R(X,Y,Z,T)$ (see \cite[Corollary XIV.1.4]{MR1666820}). Moreover, for any $x\in M_n$, $\sigma_x(M_n)=M_n$ and thus $M_n$ is a Riemannian symmetric space on its own. Now, The tangent space $\mathfrak{p}_n:=T_pM_n$ can be decomposed as $\mathfrak{p}^n_-\oplus\mathfrak{p}^n_0\oplus\mathfrak{p}^n_+$ where $\mathfrak{p}^n_-$, $\mathfrak{p}^n_0$ and $\mathfrak{p}^n_+$ satisfy properties of the proposition. We claim that for $m> n$, $\mathfrak{p}^n_-\subseteq\mathfrak{p}^m_-$ and $\mathfrak{p}^n_+\subseteq\mathfrak{p}^m_+$. Actually, if $\mathfrak{g^n}$ is the Lie subalgebra $[\mathfrak{p}_n,\mathfrak{p}_n]\oplus\mathfrak{p}_n$ of the isometry group of $M_n$ then there is a  structure of orthogonal symmetric Lie algebra (see  \cite[Chapters IV and V]{MR1834454}) on $\mathfrak{g}^n$, which can be decomposed as 
\[\mathfrak{g}^n=\mathfrak{g}^n_-\oplus\mathfrak{p}^n_0\oplus\mathfrak{g}^n_+\]
where $\mathfrak{g}^n_-,\mathfrak{g}^n_+$ are respectively compact and noncompact; and $\mathfrak{p}_0^n$ is the maximal central Abelian subspace of $\mathfrak{p}_n$. In particular, $\mathfrak{g}^n$ is a subalgebra of $\mathfrak{g}^m$ and $\mathfrak{s_n}:=\mathfrak{g}^n_-\oplus\mathfrak{g}^n_+$ is a semisimple Lie algebra and thus contained in $\mathfrak{s}_m$. The semisimple algebras $\mathfrak{s}_n$ and $\mathfrak{s}_m$ are orthogonal sums of simple ideals of compact or noncompact types. Let $\pi$ be the orthogonal projection on a simple ideal $J$ of $\mathfrak{s}_m$. The restriction of $\pi$ to any simple ideal $I$ of $\mathfrak{s}_n$ is either trivial or is an isomorphism of orthogonal symmetric Lie algebras on its image. In particular, if $\pi(I)\neq\{0\}$ then $I$ and $J$ are both compact or noncompact. This proves the claim.

 We set $\mathfrak{p}_+=\overline{\cup_n\mathfrak{p}^n_+}$, $\mathfrak{p}_-=\overline{\cup_n\mathfrak{p}^n_-}$ and $\mathfrak{p}_0=\{X\in\mathfrak{p},\ [X,Y]=0,\ \forall Y\in\mathfrak{p}\}$. Let $X\in(\mathfrak{p}_+\oplus\mathfrak{p}_-)^\bot$, then if $\pi_n\colon\mathfrak{p}\to\mathfrak{p}_n$ is the orthogonal projection on $\mathfrak{p}_n$ then $\pi_n(X)\in\mathfrak{p}^n_0$. Actually for any $Y\in\mathfrak{p}$,
\begin{align*}
[Y,X]=0&\iff [Z,[X,Y]]=0,\ \forall Z\in\mathfrak{p}\\
&\iff g([Z,[X,Y]],T)=R(X,Y,Z,T)=0,\ \forall Z,T\in\mathfrak{p}.
\end{align*}
Thus, $R(X,Y,Z,T)=\lim_nR(\pi_n(X),\pi_n(Y),Y,T)=0$ for any $Z,T\in\mathfrak{p}$ and $[X,Y]=0$.
Therefore $(\mathfrak{p}_+\oplus\mathfrak{p}_-)^\bot=\mathfrak{p}_0$ and we have the desired decomposition.
\end{proof}
\section{Nonpositive curvature}\label{nonpositive}
\subsection{Geometry of nonpositive curvature spaces}

 A Riemannian manifold of finite dimension is locally CAT(0) (or is nonpositively curved in the sense of Alexandrov) if and only if it has nonpositive sectional curvature. The same result is also true in infinite dimension and a proof can be found in \cite[Theorem IX.3.5]{MR1666820}. We refer to \cite{MR1744486} for generalities about CAT(0) spaces.

\begin{prop}\label{simplyconnected} If $(M,g)$ is a Riemannian symmetric space with nonpositive sectional curvature and no local Euclidean factor then $M$ is simply-connected, the exponential map at any point is a diffeomorphism and $M$ is CAT(0).
\end{prop}
\begin{proof}Consider the universal cover $\widetilde{M}$ of $M$. This universal cover has a natural structure of Riemannian manifold turning the projection $\pi:M\to\widetilde{M}$ into a Riemannian covering. In that way $\widetilde{M}$ is simply-connected and is locally CAT(0) since $M$ is locally CAT(0). The space $\widetilde{M}$ is a CAT(0) space thanks to Cartan-Hadamard theorem \cite[Theorem II.4.1]{MR1744486}.

Choose $\tilde{x},\tilde{y}\in\widetilde{M}$. The projection of the geodesic segment between $\tilde{x}$ and $\tilde{y}$ is a (locally minimizing) geodesic segment between $x=\pi(\tilde{x})$ and $y=\pi(\tilde{y})$. Let $f_t$ be the isometry $\sigma_{x_t}\circ\sigma_x$ where $x_t$ is the point at distance $td(\tilde{x},\tilde{y})/2$ from $x$ on the previous segment and $t\in[0,1]$. Let $(F_t)_{t\in[0,1]}$ be a lift of $(f_t)_{t\in[0,1]}$ such that $F_0=$Id. Remark that $t\mapsto F_t(\tilde{x})$ is a lift of the geodesic segment from $x$ to $y$ and since $F_0(\tilde{x})=\tilde{x}$, this is the geodesic from $\tilde{x}$ to $\tilde{y}$ and thus $F_1(\tilde{x})=\tilde{y}$ . Since $\pi$ is a Riemannian covering, we observe that $F_t$ is an isometry of $\widetilde{M}$ for any $t\in[0,1]$.

For $\gamma\in\pi_1(M)$ and $t\in[0,1]$,

\[\pi\circ F_t\circ\gamma=f_t\circ\pi\circ \gamma=f_t\circ\pi=\pi\circ F_t.\]

The map $\pi\circ F_t$ is a Riemannian covering and thus for any $t$, there exists $\gamma'$ such that $F_t\circ\gamma=\gamma'\circ F_t$. A connectedness argument shows that $\gamma'$ is independent of $t$ and since $F_0=$Id then $\gamma'=\gamma$. This shows that the displacement function of $\gamma$ is the same at $x$ and at $y$ and thus is constant on $\widetilde{M}$. Suppose this displacement length is not zero then $\gamma$ is a Clifford translation, $\widetilde{M}$ has a Euclidean factor and $\widetilde{M}\simeq\R\times \widetilde{N}$ as metric space. This is a contradiction and thus $\gamma$ is trivial.

Since we know that $M$ is simply-connected, Cartan-Hadamard theorem \cite[Theorem IX.3.8]{MR1666820} shows that the exponential map at any point is a diffeomorphism.
\end{proof}
\subsection{L$^*$-algebras associated to Riemannian symmetric spaces with nonpositive curvature operator}\label{constructionspace}
For the remainder of the section, $(M,g)$ will be a separable Riemannian symmetric with nonpositive curvature operator and no Euclidean local de Rham factor.

\begin{lemm}\label{baba}The L$^*$-algebra associated to $M$ is a Hilbertian sum $L=\mathfrak{p}_0\oplus^2_i L_i$ where $\mathfrak{p}_0$ is an abelian ideal of $\mathfrak{p}$ and each $L_i$ is a noncompact simple ideal.\end{lemm}
\begin{proof}Let $L_0$ be the center of $L$. Since $L_0$ is $*$-invariant, one can decompose $L_0=\mathfrak{p}_0\oplus\mathfrak{k}_0$ where $\mathfrak{p}_0,\mathfrak{k}_0$ are $\pm1$-eigenspaces of $*$. Since any L*-algebra is the sum of its center and a Hilbertian sum of simple ideals, one has $L=L_0\oplus_i^2L_i$ where each $L_i$ is simple. Let $\mathfrak{p}_i$ be the $1$-eigenspace of $L_i$. One has $\mathfrak{p}=\mathfrak{p}_0\oplus\oplus_i^2\mathfrak{p}_i$ and since $L=\overline{[\mathfrak{p},\mathfrak{p}]}\oplus\mathfrak{p}=\mathfrak{p}_0\oplus^2_i L_i$, one has $\mathfrak{k}_0=\{0\}$.

Assume for contradiction that there is a $L_i$ which is compact. By construction $L=\overline{[\mathfrak{p},\mathfrak{p}]}\oplus\mathfrak{p}$ and since $L_i$ is invariant under $*$ then $L_i\subset \overline{[\mathfrak{p},\mathfrak{p}]}$. Thus, $\mathfrak{p}\subseteq\oplus_{j\neq i} L_j$, $[\mathfrak{p},L_i]=0$ and $[L_i,L]=0$, which is a contradiction.
\end{proof} Thanks to the classification of simple separable real L$^*$-algebras, we know that each $L_i$ that has infinite dimension, is homothetic to one element of the list in Table \ref{liste}.

Each of these algebras can be realized as a L$^*$-subalgebra of $\mathfrak{gl}_\infty^2(\R)$, which is the Lie algebra of Hilbert-Schmidt operators of some real separable Hilbert space $\mathcal{H}$, endowed with the Hilbert-Schmidt norm. For $X\in\mathfrak{gl}_\infty^2(\R)$, $X^*$ is the adjoint of $X$ as operator on $\mathcal{H}$. The algebra $\mathfrak{gl}_\infty^2(\R)$ is the Lie algebra of the Hilbert-Lie group GL$_\infty^2(\R)$. If O$^2(\infty)$ is the intersection of GL$_\infty^2(\R)$ and the orthogonal group O$(\mathcal{H})$ of $\mathcal{H}$ then GL$_\infty^2(\R)/$O$^2(\infty)$ is a Riemannian symmetric space with nonpositive curvature operator (see for example \cite[III.2]{MR0476820}).

Let $\mathfrak{g}$ be any L$^*$-algebra of the previous list viewed as a L$^*$- subalgebra of $\mathfrak{gl}_\infty^2(\R)$. Let $G$ be the closed subgroup of GL$_\infty^2(\R)$ generated by $\exp{\mathfrak{g}}$ and $K=G\cap O(\mathcal{H})$. If $\mathfrak{g}=\mathfrak{k}\oplus\mathfrak{p}$ is the decomposition of $\mathfrak{g}$ into skew-symmetric and symmetric parts then thanks to \cite[Proposition III.4]{MR0476820}, $\exp(\mathfrak{p})$ is a totally geodesic subspace of GL$_\infty^2(\R)/$O$^2(\infty)$, $G$ acts transitively on $\exp(\mathfrak{p})$ and $K$ is the stabilizer of Id in $G$. In this way, $\exp(\mathfrak{p})\simeq G/K$. When $\mathfrak{g}$ varies among the elements of Table \ref{liste}, one obtains the irreducible symmetric spaces with nonpositive curvature operator which appear in Theorem \ref{noncompact}.

Let $L$ be a simple noncompact L$^*$-algebra , let $\mathfrak{g}$ be the element of the homothety class of $L$ that is in the previous list and let $\lambda$ be the scaling factor such that $L=\lambda\cdot\mathfrak{g}$. The \textit{Riemannian symmetric space associated to} $L$ is the space $G/K$ endowed with the metric that is the multiple by $\lambda$ of the metric coming from the embedding in  GL$_\infty^2(\R)/$O$^2(\infty)$.

It is a routine verification to show that if one starts from a simple noncompact L$^*$-algebra $L$, one considers the Riemannian symmetric space $M$ associated to $L$ and  one constructs the L$^*$-algebra as in Section \ref{construction} then the L$^*$-algebra constructed is isomorphic to $L$. 
\begin{rema}If $L$ is a noncompact simple L$^*$-algebra of finite dimension then it is a simple Lie algebra of noncompact type in the usual sense. It is associated to a Riemannian symmetric space of noncompact type and $L$ coincides with the L$^*$-algebra associated to this Riemannian symmetric space (see Example \ref{fintitedimension} and \cite[Chapter V]{MR1834454}). 
\end{rema}
\begin{proof}[Proof of Theorem \ref{noncompact}] Let  $L=\mathfrak{k}\oplus\mathfrak{p}$ be the L$^*$-algebra to $M$. This algebra $L$ is a Hilbertian sum $L=\mathfrak{p}_0\oplus^2L_i$ of simple noncompact L$^*$-algebras (Lemma \ref{baba}). For each $L_i$, let $M_i$ be the Riemannian symmetric space associated to $L_i$ and let $\mathcal{H}$ be a Hilbert space isometric to $\mathfrak{p}_0$. Now, consider the Hilbertian product $\mathcal{H}\times\prod^2_i M_i$. This is a simply-connected symmetric space whose associated L*-algebra is also $L$. Now Theorem \ref{isometry} implies that  $M$ and $\mathcal{H}\times\prod^2_i M_i$ are isometric. Since $M$ has trivial Euclidean de Rham factor, $\mathcal{H}$ is reduced to a point and  $M\simeq\prod^2_i M_i$.
\end{proof}

\begin{rema}Let $X=\prod^2_{i\in I}X_i$ and $Y=\prod^2_{j\in J}Y_i$ be two Hilbertian products of pointed metric spaces $(X_i,x_i,d_i)$ and $(Y_j,y_j,\delta_j)$. We say that $X$ and $Y$ are \textit{multihomothetic} if there exists a bijection $\varphi\colon I\to J$, a family of scaling factors $(\lambda_i)_{i\in I}$ and isometries $\Phi_i\colon(X_i,\lambda_i d_i)\to(Y_{\varphi(i)},\delta_{\varphi(i)})$ such that $\Phi_i(x_i)=y_{\varphi(i)}$.

We emphasize that the diagonal map between cartesian products
\[\begin{array}{rcc}
\Phi\colon \prod X_i&\to&\prod Y_j\\
(x_i)&\mapsto&\left(\Phi_{\varphi^{-1}(j)}(x_{\varphi^{-1}(j)})\right)
\end{array}\]
induces a bijection, which is a homeomorphism, between $X$ and $Y$ if and only if there are two positive numbers $c,C>0$ such that $c\leq\lambda_i\leq C$ for all $i\in I$.

It is a classical fact that any Riemannian symmetric space of noncompact and finite dimension is multihomothetic to  a totally geodesic subspace of SL$_n(\R)$/SO$_n(\R)$ for some $n$. This is also true in general. Let $M=\prod^2M_i$ be a separable Riemannian symmetric space with nonpositive curvature  operator and no Euclidean local de Rham factor. Let $L=\oplus^2L_i$ be its associated L$^*$-algebra. Let $\mathfrak{g}_i$ be the L$^*$-algebra homothetic to $L_i$ that is a L$^*$-subalgebra of $\mathfrak{gl}^2(\mathcal{H}_i)$ where $\mathcal{H}_i$ is a real Hilbert space of finite or infinite dimension and  $\mathfrak{gl}^2(\mathcal{H}_i)$  is the L$^*$-algebra of Hilbert-Schmidt operators on $\mathcal{H}_i$. Let $\mathcal{H}$ be the Hilbertian sum $\oplus^2\mathcal{H}_i$. Thus, 
\[\oplus^2\mathfrak{g}_i\leq\oplus^2\mathfrak{gl}^2(\mathcal{H}_i)\leq\mathfrak{gl}^2(\mathcal{H}).\]
The image by the exponential map of the symmetric part of $\oplus^2\mathfrak{g}_i$ is a totally geodesic subspace of GL$^2(\mathcal{H})/$O$^2(\mathcal{H})$ and this space is multihomothetic to $M$ (but the multihomothety is not necessarily a homeomorphism).
\end{rema}
\begin{proof}[Proof of Corollary \ref{finite rank}] Let $M\simeq\prod^2M_i$ be the Hilbertian decomposition obtained in Theorem \ref{noncompact}. Since the rank is greater or equal to the number of factors and greater or equal to the rank of each factor, the Hilbertian product is actually a finite one and each factor has finite rank. The only possible factors of infinite dimension are $X_p(\K)$ since the others contained increasing sequence of finite dimensional totally geodesic subspaces of increasing rank.

The telescopic dimension is always greater or equal to the rank and it is exactly equal to the rank when the symmetric space has finite dimension or is $X_p(\K)$ because in both cases, any asymptotic cone is a Euclidean building of dimension equal to the rank (see  \cite{MR1608566} and \cite[Corollary 1.4]{MR3044451}).
\end{proof}
\subsection{A CAT(0) symmetric space which is not a Riemannian manifold}\label{counterexample}
We describe an example of a CAT(0) symmetric space which is not a Riemannian manifold. Let $\mathbb{H}$ be  the hyperbolic plane with constant sectional curvature $-1$. We fix an origin $o\in\mathbb{H}$. We consider $X=$L$^2([0,1],\mathbb{H})$, the space of measurable maps $x\colon t\mapsto x_t$ from $[0,1]$ (endowed with the Lebesgue measure) to $\mathbb{H}$ such that $t\mapsto d(o,x_t)$ is a square integrable function. This space (called \textit{Pythagorean integral} in \cite{MR2219304}) endowed with the distance
\[d(x,y)=\left(\int_{[0,1]}d(x_t,y_t)^2\mathrm{d}t\right)^{1/2}\]
is a complete separable CAT(0) space. Geodesics can be easily described as in \textit{loc. cit.}. Actually, if $I$ is a real interval, a map $g\colon I\to X$ is a geodesic when there exist a measurable map $\alpha\colon [0,1]\to \R^+$ and a collection of geodesics $g_t\colon \alpha(t)I\to\mathbb{H}$ such that
\[\int_{[0,1]}\alpha(t)^2\mathrm{d}t=1,\ (g(s))_t=g_t(\alpha(t)s)\]
for all $s\in I$ and almost all $t\in[0,1]$. For $h\in\mathbb{H}$, let $S_h$ be the geodesic symmetry at $h$ in $\mathbb{H}$. For $x,y\in X$, we set $\sigma_x(y)$ to be the map $t\mapsto S_{x_t}(y_t)$. The description of geodesics implies that $S_x$ is the geodesic symmetry at $x$. Therefore $X$ is a CAT(0) symmetric space.

Let $X$ be a CAT(0) space and $x$ be a point in $X$. The space of directions $\Sigma_x$ of $X$ at $x$ is the set of classes of geodesic rays starting at $x$. Two rays are identified if their Alexandrov angle vanishes. The Alexandrov angle yields a distance on the quotient. The tangent cone $T_x$ is the Euclidean cone over $\Sigma_x$. We describe $\Sigma_x$ and $T_x$ for $x\in$L$^2([0,1],\mathbb{H})$ below. We denote by $\overline{\angle}_x(y,z)$  the comparison angle and by $\angle_x(y,z)$ the Alexandrov angle at $x$ between $y$ and $z$.
\begin{defi}Let $(Y,d)$ be a separable metric space of diameter less than $\pi$ and $(\Omega,\mu)$ a standard measure space. The \emph{integral join}, $\int_\Omega^*Y$, is the set of pairs $(y,v)=((y_\omega),(v_\omega))$ such that
\begin{enumerate}[(i)]
\item for all $\omega\in\Omega$, $y_\omega\in Y$ and $v_\omega\in\mathbb{R}^+$,
\item the map $\omega\mapsto v_\omega$ is measurable and $\int_\Omega v^2_\omega\mathrm{d}\mu(\omega)=1$, 
\item the map $\omega\mapsto y_\omega$ is measurable.
\end{enumerate} 
The metric on $\int_\Omega^*Y$ is defined by the formula
\[\cos\left(d((x,v),(y,w))\right)=\int_\Omega v_\omega w_\omega \cos(d(x_\omega,y_\omega))\mathrm{d}\mu(\omega).\]
\end{defi}
Let $\Sigma_o$ be the space of directions at our base point $o\in\mathbb{H}$. The tangent cone $T_o$ is simply the tangent space at $o$ and thus isometric to $\mathbb{R}^2$.
\begin{prop}(see also \cite[Remark 48]{MR2219304}) Let $x$ be a point in $L^2([0,1],\mathbb{H})$. The space of directions at $x$ is isometric to $\int_{[0,1]}^*\Sigma_o$. The tangent cone at $x$ is isometric to the Pythagorean integral $L^2([0,1],T_o)$ which is a Hilbert space.
\end{prop}
\begin{proof}Let $g,g'$ be two geodesics rays of $L^2([0,1],\mathbb{H})$ starting at $x$. Thanks to the description of geodesics, there exist $\{g_t\},\{g_t'\}$, families of geodesic rays starting at $o$ in $\mathbb{H}$ and $v,v'$ measurable maps $[0,1]\to\R^+$ with $L^2$-norm equal to 1. Therefore,
\begin{align*}
\cos(\angle_x(g,g))&=\lim_{s\to0}\cos(\cpangle_x(g(s),g(s)))\\
&=\lim_{s\to0}\frac{2s^2-d(g(s),g'(s))^2}{2s^2}\\
&=1-1/2\lim_{s\to0}\frac{d(g(s),g'(s))^2}{s^2}\\
&=1-1/2\lim_{s\to0}\frac{1}{s^2}\int_t(v_t^2+{v'}_t^2)s^2-2v_t v_t'\cos(\cpangle_o(g_t(v_t s),g'_t(v'_t s)))dt\\
&=\int_t v_t v'_t \cos(\angle_o(g_t,g'_t))dt .\\
\end{align*}
This equality shows that  $\Sigma_x$ embeds isometrically in  $\int^\ast_{[0,1]}\Sigma_{o}$. Conversely, if $((g_t),(v_t))$ is an element in  $\int^\ast_{[0,1]}\Sigma_{o}$, one can construct the geodesic $s\mapsto g(s)$ where $(g(s))_t=g_t(v_t s)$ for almost every $t$.\\

Now, we define a map  $\Phi\colon T_x\to L^2({[0,1]},T_o)$ through the formula
\[(\lambda,(g_t,v_t))\mapsto(\lambda v_t,g_t).\]
We compute
\[
d((\lambda,(g,v)),(\lambda',(g',v')))^2=\lambda^2+{\lambda'}^2-2\lambda\lambda'\int_{[0,1]} v_t v'_t \cos(\angle_o(g_t,g'_t))dt
\]
and
\begin{align*}
d((\lambda v_t,g_t),(\lambda' v'_t,g'_t))^2&=\int_{[0,1]}(\lambda v_t)^2+(\lambda'v'_t)^2-2\lambda v_t\cdot\lambda' v'_t\cos(\angle_o(g_t,g'_t))dt\\
&=\lambda^2+\lambda'^2-2\lambda\lambda'\int_{[0,1]} v_t v'_t \cos(\angle_o(g_t,g'_t))dt.
\end{align*}
This shows that $\Phi $ is an isometry and its inverse is given by 
\[(\lambda_t,g_t)\mapsto\left(\lambda,(g_t,\lambda_t/\lambda)\right)\]
where $\lambda=\sqrt{\int\lambda_t^2dt}$.
\end{proof}

A notion of \textit{bounded curvature} for geodesic metric spaces has been introduced in \cite{MR1373605}. We give a slightly different definition but equivalent in the case of CAT(0) symmetric spaces. If $x,y,z$ are distinct points in a CAT(0) space, we denote the area of the comparison Euclidean triangle by $S_{x,y,z}$.

\begin{defi}A CAT(0) space $X$ has \emph{bounded curvature} if for any $p\in X$, there exist $\rho_p,\mu_p>0$ such that for $x,y,z\in B(p,\rho_p)$, $y'\in]x,y]$ and $z'\in]x,z]$ we have
\[|\overline{\angle}_x(y,z)-\overline{\angle}_x(y',z')|\leq \mu_pS_{x,y,z}.\]
\end{defi}
In the case where $d(x,y)=d(x,z)=r$ then $S_{x,y,z}=\overline{\angle}_x(y,z)\frac{r^2}{2}$ and the condition of bounded curvature is 
\[\left|1-\frac{\angle_x(y,z)}{\overline{\angle}_x{y,z}}\right|\leq\frac{\mu_xr^2}{2}.\]
Since we restrict our definition of bounded curvature to CAT(0) spaces, it is actually a lower bound condition on the curvature. This condition is a local condition. If $M$ is a Riemannian manifold with nonpositive sectional curvature and with locally a uniform lower bound on the sectional curvature then $M$ has bounded curvature. This is a consequence of Rauch comparison theorem \cite[Theorem XI.5.1]{MR1666820}. In  particular, any Riemannian symmetric space of nonpositive sectional curvature has bounded curvature. Since these spaces are homogeneous, the lower bound of the sectional curvature at any point is actually a global lower bound. Observe that a tree with a vertex of valency larger than 2 does not have bounded curvature.

\begin{prop}The space $L^2([0,1],\mathbb{H})$ is not a Riemannian manifold.
\end{prop}
\begin{proof}
It suffices to show that $X=L^2([0,1],\mathbb{H})$ does not have bounded curvature. We fix $r<0$, $\alpha\in(0,\pi)$ and two geodesic rays starting at $o$ with an angle equal to $\alpha$ at $o$. For $0<\lambda<1$, we set $x_1^\lambda=\rho_1(r/\lambda)$ and $x_2^\lambda=\rho_2(r/\lambda)$. We construct points $x,y^\lambda,z^\lambda\in X$ defined by
\begin{align*}
x_t&=o\ \mathrm{for}\ t\in[0,1],\\
y_t^\lambda&=o\ \mathrm{for}\ t\in(\lambda,1],\\
z_t^\lambda&=o\ \mathrm{for}\ t\in(\lambda,1],\\
y_t^\lambda&=x_1^\lambda\ \mathrm{for}\ t\in[0,\lambda],\\
z_t^\lambda&=x_2^\lambda\ \mathrm{for}\ t\in[0,\lambda].
\end{align*} 
We have $d(x,y^\lambda)=d(x,z^\lambda)=r$ and $\overline{\angle}_x(y^\lambda,z^\lambda)=\overline{\angle}_o(x_1^\lambda,x_2^\lambda)$ which tends to $\pi$ as $\lambda\to0$. Since $\angle_x(y^\lambda,z^\lambda)=\angle_o(x_1^\lambda,x_2^\lambda)=\alpha$; choosing $\alpha$ small enough, the bounded curvature condition is not satisfied.
\end{proof}

If two geodesic rays starting at a point $x\in X=$L$^2([0,1])$ have vanishing Alexandrov angle then they are actually contained one in another. This allows us to define an exponential map $\exp_x\colon T_x\to X$. If $v\in T_x$ then $\exp_x(v)$ is defined to be the point at distance $||v||$ from $x$ in the direction corresponding to $v$. This map is a bijection and its inverse is continuous but the same example as above shows that $\exp_x$ is not continuous.
\begin{rema}This space has long been known and one can find a similar space denoted $\mathcal{H}^0(M,M')$ on p.134 of \cite{MR0164306}. The authors claimed that this space is not a manifold.
\end{rema}
\begin{rema} It has been proved in \cite[Proposition 3.9]{mathese} that a CAT(0) symmetric space with bounded curvature and no branching geodesics is homeomorphic to a Hilbert space. More precisely, an exponential map is defined from the tangent cone to the space  and this exponential map is a homeomorphism.
\end{rema}
\section{Nonnegative curvature}\label{nonnegative}

\begin{proof}[Proof of Theorem \ref{nonnegativebis}] By construction, for any $x\in L,  x^*=-x$ and if $L=L_0\oplus_i^2L_j$ where $L_0$ is the center and each $L_j$ is a simple ideal. Observe that each $L_i$ is compact. Let $L=L_0\oplus^2 I_i$ be the decomposition of $L$ into abelian and irreducible ideals invariant under $\theta$. That is, each $I_i$ is a simple ideal invariant under $\theta$ or $I_i=L_{i_1}\oplus L_{i_2}$ where $L_{i_1},L_{i_2}$ are simple ideals interchanged by $\theta$ (Proposition \ref{compact}). Let $\mathcal{H}$ be a Hilbert isometric to $I_0$ and assume that for all $I_i$, there is a simply-connected Riemannian symmetric space $M_i$ whose associated L*-algebra is $I_i$. The product space $\mathcal{H}\times \prod^2_i M_i$ is a simply-connected Riemannian symmetric space whose associated L*-algebra is $L$. By Theorem \ref{isometry}, $M$ and $\mathcal{H}\times\prod^2_i M_i$ are isometric.

Now, It remains to find a (unique) simply-connected Riemannian symmetric space for each compact  irreducible symmetric orthogonal L$^*$-algebra. We start with the case of an irreducible symmetric orthogonal L$^*$-algebra $(L,s)$ such that $L=\mathfrak{g}\oplus\mathfrak{g}$ where $\mathfrak{g}$ is a simple compact L$^*$ algebra (see Table \ref{listec}) and $s(X,Y)=(Y,X)$. For each $\mathfrak{g}=\mathfrak{u}^2(\infty),\mathfrak{o}^2(\infty)$ and $\mathfrak{sp}^2(\infty)$, we consider the Hilbert-Lie group $G=U^{2}(\infty), SO^2(\infty)$ (the identity component of $O^2(\infty)$, that is the set of operators in $O^2(\infty)$  such that $-1$ as infinite or even multiplicity as eigenvalue) and $Sp^2(\infty)$. The Lie algebra of $G$ is exactly $\mathfrak{g}$. We endow $G$ with the Riemannian metric induced by the scalar product on $\mathfrak{g}$ and invariant under left and right multiplication (observe that this metric is actually invariant under conjugation). Thus $G\times G$ acts by isometries on $G$ via the formula $(g,h)\cdot k=gkh^{-1}$ and $G$ identifies with $G\times G/\Delta G$ where $\Delta G$ is the diagonal subgroup in $G\times G$. The involution $g\mapsto g^{-1}$ is an isometry too with $id$ as isolated fixed point. So $G$ is a Riemannian symmetric space and $(\mathfrak{g},s)$ is the symmetric orthogonal L$^*$ algebra associated to $G$.

The remaining cases (of 3 types) will be realized as (universal cover of) homogeneous spaces $G/K$ where $G$ is one of the above Hilbert-Lie group associated to a simple compact L$^*$-algebra and $K$ will be the connected component of the fixed points set of an isometric involution $\sigma$ which yields the symmetry at $K$ in the symmetric space $G/K$.

\begin{enumerate}

\item Let $\mathcal{H}$ be a separable infinite dimensional real (resp. complex, quaternionic) Hilbert space with an orthogonal decomposition $\mathcal{H}=\mathcal{H}_1\oplus\mathcal{H}_2$ where $\mathcal{H}_2$ is infinite dimensional and $\mathcal{H}_1$ has dimension $p\in\mathbb{N}\cup\{\infty\}$. Let $J$ be the linear map $-id_{\mathcal{H}_1}\oplus id_{\mathcal{H}_2}$ and $\sigma(g)=JgJ$ for $g\in SO^2(\infty)$ (resp. $U^2(\infty)$ and $Sp^2(\infty)$). In this case, the connected component of the fixed points set of $\sigma$ is $SO^2(p)\times SO^2(\infty)$ (resp. $U^2(p)\times U^2(\infty)$ and $Sp^2(p)\times Sp^2(\infty)$).

\item Let $\mathcal{H}$ be a separable infinite dimensional real (resp. complex) Hilbert space and let $J$ be a complex (resp. quaternionic) structure on $\mathcal{H}$. That is $J$ corresponds to the multiplication by $i\in\mathbb{C}$ (resp. by the quaternionic number $j$). The involution $\sigma$ on $SO^2(\infty)$ (resp. $U^2(\infty)$) is $g\mapsto JgJ^{-1}$. In this case, the connected component of fixed points set is $U^2(\infty)$ (resp. $Sp^2(\infty)$).

\item Let $\mathcal{H}$ be a complex (resp. quaternionic) separable Hilbert space of infinite dimension and $\mathcal{H}_0$ be a real (resp. complex) form of $\mathcal{H}$ i.e. $\mathcal{H}=\mathcal{H}\oplus i\mathcal{H}_0$ (resp. $\mathcal{H}=\mathcal{H}\oplus j\mathcal{H}_0$ and let $J$ be the $\mathbb{R}$-linear (resp. $\mathbb{C}$-linear) map $id_{\mathcal{H}_0}\oplus -id_{i\mathcal{H}_0}$ (resp. $id_{\mathcal{H}_0}\oplus -id_{j\mathcal{H}_0}$). The involution $\sigma$ on $U^2(\infty)$ (resp. $Sp^2(\infty)$) is $g\mapsto JgJ^{-1}$. In this case, the connected component of fixed points set is $SO^2(\infty)$ (resp. $U^2(\infty)$).\end{enumerate}
\end{proof}

\begin{rema}The Riemannian symmetric spaces $Sp^2(\infty), Sp^2(\infty)/U^2(\infty), Sp^2(p+\infty)/\\ Sp^2(p)\times Sp^2(\infty), SO^2(p+\infty)/SO^2(p+\infty), SO^2(\infty)/U^2(\infty)$ and $ U^2(p+\infty)/ U^2(p)\times U^2(\infty)$ are simply-connected. The Riemannian symmetric spaces $U^2(\infty), U^2(\infty)/Sp^2(\infty)$ and $U^2(\infty)/SO^2(\infty)$ have fundamental group $\mathbb{Z}$ and $SO^2(\infty)$ has fundamental group $\mathbb{Z}/2\mathbb{Z}$. One can look at Chapter III  and II.8 of \cite{MR0476820} for more details.
\end{rema}

\section{Boundedness of the curvature operator}\label{curvatureoperator}
Let $M$ be a Riemannian symetric space let and $p$ be a point in $M$. We defined the \emph{curvature operator} $C\colon\Wedge\mathfrak{p}\to$L$(\mathfrak{p})$ (see Definition \ref{curvop}) thanks to the relation
\[\langle C(X\wedge Y),Z\wedge T\rangle_g=R(Y,X,Z,T).\]
A natural question, is to know if the symmetric bilinear form on $\Wedge\mathfrak{p}$ defined by
\[(X\wedge Y,Z\wedge T)=R(Y,X,Z,T)\]
is bounded. A positive answer to this question would imply that the curvature operator $C$ can be identified with a bounded operator from $\overline{\Wedge\mathfrak{p}}$, the completion of $\Wedge\mathfrak{p}$, to itself. In this case, the answer to the question asked in Question\ref{question} is positive. Actually, one can use the spectral theorem to the curvature operator (which is symmetric) on the Hilbert $\overline{\Wedge\mathfrak{p}}$ to decompose this Hilbert in an orthogonal sum of nonpositive and nonnegative part of $C$. Working a little bit harder, one can deduce a decomposition of $\mathfrak{p}$ as in Question \ref{question}.

Unfortunately the answer is negative in general even if the symmetric space has finite rank. Let us describe some examples.

\begin{exem}For spaces of constant sectional curvature $\kappa$, that are Hilbert spaces, spheres and hyperbolic spaces, the curvature operator is simply an homothety or ratio $\kappa$, that is for any $U\in\Wedge\mathfrak{p}$, $C(U)=\kappa U$ (under the identification of $\Wedge\mathfrak{p}$ with a subspace of L$(\mathfrak{p})$). References are \cite[Proposition 5, Section 3.3]{MR2243772} and \cite[Remark 1, IX, \S 3]{MR1666820}. In particular, the curvature operator is bounded.
\end{exem}

\begin{exem}\label{unbounded}Consider the symmetric space GL$_\infty^2(\R)/$O$^2(\infty)$ which can be identified with the set of positive definite operator $A$ such that $A-I$ is a Hilbert-Schmidt operator. The tangent space at the identity can be identified with the space of  symmetric Hilbert-Schmidt operators on a separable Hilbert space $\mathcal{H}$. We denote this space by $\mathfrak{s}^2(\mathcal{H})$ and by $\mathfrak{o}^2(\mathcal{H})$ the space of skew-symmetric Hilbert-Schmidt operators. The L$^*$-algebra associated to GL$_\infty^2(\R)/$O$^2(\infty)$ is $\mathfrak{gl}^2(\mathcal{H})=\mathfrak{s}^2(\mathcal{H})\oplus\mathfrak{o}^2(\mathcal{H})$. With our general notations $\mathfrak{p}=\mathfrak{s}^2(\mathcal{H})$, $\mathfrak{k}=\mathfrak{o}^2(\mathcal{H})$ and the Lie bracket is the usual bracket between operators. If $(e_i)$ is a orthonormal base of $\mathcal{H}$ then the maps $E_{i,j}\colon x\to \langle x,e_j\rangle e_i$ is an orthonormal base for $\mathfrak{gl}^2(\mathcal{H})$. We defined $S_{ij}=E_{i,j}+E_{j,i}$ and $A_{i,j}=E_{i,j}-E_{j,i}$ that yield respectively orthogonal bases for $\mathfrak{s}^2(\mathcal{H})$ and $\mathfrak{o}^2(\mathcal{H})$. Simple computations based on 
\[[S_{i,j},S_{k,l}]=\delta_{j,k}A_{i,l}+\delta_{j,l}A_{i,k}+\delta_{i,k}A_{j,l}+\delta_{i,l}A_{j,k}\]
and 
\[\trace(A_{i,j}A_{k,l})=2(\delta_{j,k}\delta_{i,l}-\delta_{j,l}\delta_{i,k})\]
show that
\begin{align*}
R(S_{i,j},S_{k,l},S_{m,n},S_{p,q})&=\trace\left([S_{i,j},S_{k,l}][S_{m,n},S_{p,q}]\right)\\
&=2\left(\delta_{j,l}\delta_{n,q}(\delta_{k,m}\delta_{i,p}-\delta_{k,p}\delta_{i,m})+\delta_{i,k}\delta_{m,p}(\delta_{l,n}\delta_{j,q}-\delta_{l,q}\delta_{j,n})\right).
\end{align*}
In particular for $i,j,k,l$ distincts
\begin{equation}\label{truc}(S_{i,j}\wedge S_{j,l},S_{i,k}\wedge S_{k,l})=R(S_{i,j},S_{j,l},S_{i,k},S_{k,l})=-2.\end{equation}
Since $(S_{i,j}\wedge S_{k,l})_{i,j,k,l}$ is an orthogonal base for $\overline{\Wedge\mathfrak{p}}$; if there was an operator (even an unbounded one) $C\colon \Wedge\mathfrak{p}\to\overline{\Wedge\mathfrak{p}}$ then  the vector $C(S_{i,j}\wedge S_{j,l})$ would have infinitely many coordinates equal to -2, which is impossible in a Hilbert space.
\end{exem}
Actually, Example \ref{unbounded} comes from the fact that if $A$ is a Hilbert-Schmidt operator on a Hilbert space $\mathcal{H}$ then ad$(A)$ is not necessarily a Hilbert-Schmidt operator on $\mathfrak{gl}^2(\mathcal{H})$.
\begin{exem}Consider $O^2(p,\infty)/O^2(p)\times O^2(\infty)$ as a totally geodesic submanifold of GL$_\infty^2(\R)/$O$^2(\infty)$ (see Section \ref{constructionspace}). This space as rank $p$ and its tangent space at the identity has orthogonal base $(S_{i,j})$ with $i\leq p$ and $j>p$. If $p\geq2$, the same computation as in Equation (\ref{truc}) with $i,l\leq p$ and $k,j>p$ leads to the same conclusion, that is the curvature operator does not come from an operator on $\Wedge(\mathfrak{p})$.
\end{exem}

\bibliographystyle{cdraifplain}
\bibliography{../../../Latex/Biblio/biblio.bib}
\end{document}